\documentclass[11pt]{amsart}

\usepackage{verbatim, amssymb,hyperref}

\usepackage{accents}
\usepackage{color}
\usepackage{soul}
\usepackage{bbm}

\newcommand{\black}{\color[rgb]{0,0,0}}
\newcommand{\red}{\color[rgb]{1,0,0}}

\newcommand{\grey}{\color[rgb]{.7,.7,.7}}

\newcommand{\eps}{\varepsilon}

\newcommand{\einschraenkung}{\,\rule[-5pt]{0.4pt}{12pt}\,{}}
\newcommand{\sprod}[1]{\langle #1 \rangle}

\newcommand{\cF}{\mathcal F}
\newcommand{\cC}{\mathcal C}
\newcommand{\cD}{\mathcal D}
\newcommand{\cE}{\mathcal E}

\newcommand{\cH}{\mathcal H}
\newcommand{\cB}{\mathcal B}

\newcommand{\cO}{\mathcal O}

\newcommand{\cU}{\mathcal U}

\newcommand{\C}{C}

\newcommand{\grad}{\nabla}

\newcommand{\dd}{\mathrm{d}}
\newcommand{\sfrac}[2]{\mbox{$\frac{#1}{#2}$}}

\newcommand{\II}{\mathbb I}

\newcommand{\cov}{\mathrm{cov}}

\newcommand{\1}{1\hspace{-0.098cm}\mathrm{l}}

\newcommand{\id}{\mathrm{id}}

\renewcommand{\P}{{\mathbb P}}

\newcommand{\N}{{\mathbb N}}
\newcommand{\IA}{{\mathbb A}}
\newcommand{\E}{{\mathbb E}}

\newcommand{\R}{{\mathbb R}}

\newcommand{\cN}{{\mathcal N}}

\newcommand{\IU}{{\mathbb U}}
\newcommand{\Hess}{\text{Hess}\,}

\newcommand{\stab}{\stackrel{\mathrm{stably}}\Longrightarrow}

\newcommand{\Uconv}{\IU^{\mathrm{conv}}}
\newcommand{\Mconv}{\mathbb M^\mathrm{conv}}

\theoremstyle{plain}
\newtheorem{theorem}{Theorem}[section]
\newtheorem{prop}[theorem]{Proposition}
\newtheorem{lemma}[theorem]{Lemma}
\newtheorem{cor}[theorem]{Corollary}
\newtheorem{defi}[theorem]{Definition}

\theoremstyle{definition}
\newtheorem{rem}[theorem]{Remark}

\begin{document}

	\title[CLTs for averaged stochastic gradient descent]%
	{Central limit theorems for stochastic gradient descent with averaging for stable manifolds}

	\author[]
	{Steffen Dereich}
	\address{Steffen Dereich\\
		Institut f\"ur Mathematische Stochastik\\
		Fachbereich 10: Mathematik und Informatik\\
		Westf\"alische Wilhelms-Universit\"at M\"unster\\
		Orl\'eans-Ring 10\\
		48149 M\"unster\\
		Germany}
	\email{steffen.dereich@wwu.de}
	\thanks{Funded by the Deutsche Forschungsgemeinschaft (DFG, German Research Foundation) under Germany's Excellence Strategy EXC 2044 –390685587, Mathematics Münster: Dynamics–Geometry–Structure.}
	
	\author[]
	{Sebastian Kassing}
	\address{Sebastian Kassing\\
		Institut f\"ur Mathematische Stochastik\\
		Fachbereich 10: Mathematik und Informatik\\
		Westf\"alische Wilhelms-Universit\"at M\"unster\\
		Orl\'eans-Ring 10\\
		48149 M\"unster\\
		Germany}
	\email{sebastian.kassing@wwu.de}

	
	\begin{abstract}
		In this article we establish new central limit theorems for Ruppert-Polyak averaged stochastic gradient descent  schemes. Compared to previous work we do not assume that convergence occurs to an isolated attractor but instead allow convergence to a stable manifold. On the stable manifold the target function is constant and the oscillations in the tangential direction may be significantly larger than the ones in the normal direction. As we show, one still recovers a central limit theorem with the same rates as in the case of isolated attractors. Here we consider step-sizes $\gamma_n=n^{-\gamma}$ with  $\gamma\in(\frac34,1)$, typically. 
		%
		%
	\end{abstract}

	\maketitle
	
	\black
	\section{Introduction}\label{s1}

We consider stochastic gradient descent (SGD) algorithms for the approximation of minima of functions
$-F: \R^d \to \R$, where, at each point $x \in \R^d$, we are only able to simulate  a noisy version of the gradient $f(x)= DF(x)$.

Stochastic approximation methods form  a popular class of optimisation algorithms with applications in diverse areas of statistics, engineering and computer science. Nowadays a key application lies in machine learning where it is used in the training of neural networks.    The original concept  was introduced 1951  by Robbins and Monro \cite{RM51} and since then analysed in various directions. The optimal order of convergence id obtained for step-sizes of order $C\frac 1n$ with $C>0$ being an approriate constant depending in a nontrivial way on the problem.  As found by Ruppert \cite{Rup82} and Polyak \cite{Pol90,PJ92}  the running average of a Robbins-Monro algorithm yields the optimal order of convergence even in the case of slower decaying step-sizes. 
Following these original papers a variety of results were derived and we refer the reader to the monographs by  \cite{BMP90,Duf96,KY03} for more details.

In previous research, a typical key assumptions is that $-F$ has \emph{isolated} local minima and is (at least locally) strictly convex around these. These assumptions are often not met in practice and as an example we outline  an application from machine learning \cite{vidal2017mathematics}. 
In a neural network with ReLU activation function the positive homogeneity of the activation function entails that every (representable) function possesses a non-discrete set of representations as deep learning network. In this context it appears natural  to ask for extensions of classical research on settings where the set of (local) minima forms a stable manifold. So far research in that direction is very limited. 
 Fehrman et al. \cite{FGJ2019} establish rates for the convergence of the target function of a stochastic gradient descent  scheme under the assumption that the set of minima forms   a  stable manifold. Tripuraneni and Flammarion  \cite{tripuraneni2018averaging} devise an averaging-method on submanifolds so that the Ruppert-Polyak result is applicable for the approximation of an isolated minimum of a function $f$ defined on a Riemannian manifold $M$. Li and Yuan \cite{li2017convergence} show convergence to a unique teacher network in the setting of two-layer feedforward networks with ReLU activations and identity mappings in two phases. 
If the SGD scheme does not escape to infinity  the gradient of the target function at the random evaluation point tends to zero. This implies convergence of the SGD scheme in the case where $\{x\in \R^d: DF(x)=0\}$ consists of isolated points, see \cite{bertsekas2000gradient, ghadimi2013stochastic, li2018convergence, lei2019stochastic}. 



Let us introduce the central dynamical system considered in this article. Let $(\Omega, \cF, (\cF_n)_{n \in \N_0},\P)$ be a filtered probability space and $F:\R^d \to \R$ a measurable and differentiable function and set $f=DF:\R^d\to \R^d$. Let $M$ be a $d_\zeta$-dimensional $C^1$-submanifold of $\R^d$ with 
$$
	f\big|_M \equiv 0.
$$
We consider an adapted dynamical system $(X_n)_{n \in \N_0}$ satisfying for all $n \in \N$
\begin{align} \label{eq:dynsys1}
	X_n = X_{n-1}+\gamma_n (f(X_{n-1})+D_n),
\end{align}
where
\begin{enumerate}
\item[(0)] $X_0$ is a $\cF_0$-measurable $\R^d$-valued random variable, the \emph{starting value},
	\item[(I)] $(D_n)_{n\in\N}$ is an $\R^d$-valued, adapted process, the \emph{perturbation}, 
	\item[(II)] $(\gamma_n)_{n\in\N}$ is a sequence of strictly positive reals, the \emph{step-sizes}.
\end{enumerate}
We briefly refer to $(X_n)$ as the \emph{Robbins-Monro system}.
Furthermore, we consider  for $n\in\N$  the \emph{Ruppert-Polyak average with  burn-in} given by
\begin{align}\label{eq:PR1}
	\bar X_n = \frac{1}{\bar b_n}\sum_{i=n_0(n)+1}^{n} b_i X_i,
\end{align}
where  
\begin{enumerate}
\item[(III)] $(n_0(n))_{n\in\N}$ is a $\N_0$-valued sequence with $n_0(n)<n$ for all $n\in\N$ and $n_0(n) \to \infty$,
\item[(IV)] $(b_n)_{n \in \N}$ is a sequence of strictly positive reals and $\bar b_n= \sum_{i=n_0(n)+1}^{n}b_i$ for $n\in\N$.
\end{enumerate}

Roughly speaking, we raise and (at least partially) answer the following questions.
\begin{itemize}
\item Is Ruppert-Polyak averaging still  beneficial in the case of non-isolated minimizers?
\item If so, what are good choices for the parameters introduced in (II) to (IV)?
\end{itemize}
We answer these questions by deriving  central limit theorems for the performance of the Ruppert-Polyak average on the event of convergence of $(X_n)$ to some element of  the stable manifold $M$.

Let us be more precise.  By assumption $M$ is a $C^1$-manifold and we will impose additional regularity assumptions on  the tangent spaces (see Definition~\ref{def:Phi}) that will guarantee existence of an open neighbourhood $\mathbb M$ of $M$ so that for each $x\in \mathbb M$ there exists a unique closest element~$x^*$ in $M$, the \emph{$M$-projection of $x$} (cf. \cite{dudek1994nonlinear}, \cite{leobacher2018existence}).    We denote by $\Mconv$ the event that $(X_n)$ converges to an element of $M$ and denote the limit by $X_\infty$. Note that on $\Mconv$  the  $M$-projection $X_n^*$ and $\bar X_n^*$ are well-defined for sufficiently large (random) $n$ and we will provide stable limit theorems for
$$
\sqrt n (\bar X_n-\bar X_n^*) \text{ \ and \ } n\, (F(X_\infty)- F(\bar X_n)),
$$
on the event $\Mconv$. 
Our analysis is conducted in a very general setup. However, we will make our findings transparent in the particular case, where the perturbation is  a sequence of  square integrable martingale differences whose conditional covariance converges to a random matrix $\Gamma$, a.s., on $\Mconv$. Here we prove that under appropriate assumptions to be found in Theorem~\ref{thm:MainCLT_special} the Ces\`aro average
$$
\bar X_n=\frac1{n-n_0(n)}\sum_{k=n_0(n)+1}^n X_k
$$
converges in the stable sense, on  $\Mconv$, 
	\begin{align*}
		 \sqrt n \ (\bar X_n-\bar X_n^*)  \stackrel{\mathrm{stably}}\Longrightarrow   \bigl (Df(X_\infty)\big|_{N_{X_\infty} M}\bigr)^{-1} \ \Pi_{N_{X_\infty}M} \ \cN (0, \Gamma), 
	\end{align*}
where the right hand side stands for the random distribution obtained when applying the orthogonal projection $\Pi _{N_{X_\infty}M}$ onto the normal space of $M$ at $X_\infty$ and the inverse of the restricted random mapping $Df(X_\infty)\big|_{N_{X_\infty} M}:N_{X_\infty} M\to N_{X_\infty} M$ (which will exist as consequence of a variant of the standard contractivity assumption) to a centered Gaussian random variable with covariance $\Gamma$. Note that the order of convergence is the same as for isolated attractors. Moreover, in the latter case the manifold $M$ is zero dimensional and $N_{X_\infty} M=\R^d$ so that one recovers the classical result that, on $\Mconv$,
	\begin{align*}
		 \sqrt n \ (\bar X_n-X_\infty)  \stackrel{\mathrm{stably}}\Longrightarrow   \bigl (Df(X_\infty)\bigr)^{-1} \  \cN (0, \Gamma). 
	\end{align*}
  
Still there is a crucial difference between the setting with isolated attractors and the one we discuss here. To explain this and later to do the proofs, we assume existence of particular local manifold representations $\Psi: U \to \R^d$ around some open sets $U\subset \R^d$ which allow us to associate every  $x\in M\cap U$ with  coordinates $$\Psi(x)=\begin{pmatrix} \Psi_\zeta(x)\\ \Psi_\theta(x)\end{pmatrix} \in \R^{d_\zeta}\times \{0\}^{d_\theta}\subset \R^d$$ in such a way that for $x\in U$
$$
\Psi(x^*) = \begin{pmatrix} \Psi_\zeta(x)\\ 0 \end{pmatrix}.
$$ 
In the representation we thus have well separated directions. The tangential directions are the ones in $\R^{d_\zeta}\times \{0\}^{d_\theta}$ and the normal ones are the ones in $\{0\}^{d_\zeta}\times \R^{d_\theta}$ with $d_\theta=d-d_\zeta$.
On the event that $(X_n)$ converges to some element of $U\cap M$ the sequence has all but finitely many entries in $U$. In the new coordinates the fluctuations in the normal direction will behave as in the classical theory whereas the fluctuations in the tangential direction are typically \emph{larger} since there is no restoring force acting in this direction. This explains why we need to compare $\bar X_n$ with $\bar X_n^*$ and not $X_\infty$ in the central limit theorem. The fluctuations in the tangential direction do not appear in the limit distribution, but we will  impose additional assumptions on the sequence of step-sizes to show that these effects are negligible. More explicitly, in the setting with the highest regularity we allow step-sizes $\gamma_n=n^{-\gamma}$ with $\gamma\in(\frac 34,1)$. In the case of isolated attractors one typically allows exponents $\gamma\in(\frac 12 ,1)$. 

	We proof central limit theorems for this more general situation following the martingale CLT approach introduced in \cite{Sac58}. Our main results are stated in Section~\ref{sec:theorems}. The proofs are based on various preliminary considerations that are carried out in Sections~\ref{sec:geometry} to~\ref{sec:prop}. These steps can be roughly summarized as follows.
	In Section~\ref{sec:geometry} we provide some basic geometric essentials about the involved manifolds.  In Section~\ref{sec:L2conv}, we derive an $L^2$-bound for the order of convergence of the distance $d(X_n,M)$ which will be a main tool to control certain error terms caused by a linear approximation. In Section~\ref{sec:linSys},  we analyse a related linear system. 
	In Section~\ref{sec:prop}, we analyse various error terms that will appear in our main proofs. 
Section~\ref{sec:proof} is devoted to the proofs of the main theorems. 
The main theorems use stable convergence \emph{restricted to  sets} that are not necessarily  almost sure sets. The respective notion of convergence is introduced and analysed in detail in Section~\ref{sec:appendix}.

In the article, we use $\cO$-notation. For a multivariate function $(f_n)$ and a strictly positive function function $(g_n)$ we write
$$
f_n=\cO(g_n) \text{ \ \ if and only if \ \ }\sup_{n} \frac{|f_n|}{g_n}<\infty
$$
and $$
f_n=o(g_n) \text{ \ \ if and only if \ \ }\lim_{n\to\infty} \frac{|f_n|}{g_n}=0
$$
with the former notation making sense for arbitrary domains and the latter one for domains being subsets of $\R$.
We also make use of the notation in a probabilistic sense, see Section~\ref{sec:8_2} for details.

	\section{The central limit theorem} \label{sec:theorems}

	In this section we introduce the main result of the article,  a central limit theorem for the averaged Robbins-Monro scheme on $\Mconv$. We start with introducing the central definitions.
		Generally, we denote for a $C^1$-submanifold $M\subset \R^d$, by $T_xM$ the tangent space of $M$ at $x\in M$ and by $N_x M=(T_xM)^{\bot}$ the normal space of $M$ at $x$.
	
	\begin{defi}   A pair $(F,M)$ consisting of a differentiable function $F:\R^d\to \R$ and a $d_\zeta$-dimensional $C^1$-submanifold $M$ of $\R^d$ is called \emph{approximation problem} if the following holds
		\begin{enumerate} \item [(i)]
			$	\displaystyle{	DF\big|_M\equiv 0 }$
			\item[(ii)]   
			
			$f=DF$ is continuously differentiable on $M$ and
			\item [(iii)] for every $x\in M$, the differential $Df(x)$ is symmetric and satisfies,  for every $v\in N_xM\backslash \{0\}$,
			\begin{align} \label{eq:F-attractor2}
			\langle v, Df(x) v\rangle <0.
			\end{align}
		\end{enumerate}	
		Set $d_\theta= d- d_\zeta$.
	\end{defi}
	
	\begin{rem}\label{rem:8456}
		If $(F,M)$ is an approximation problem, then for every $x\in M$ the symmetric matrix $Df(x)$ admits an orthonormal basis of eigenvectors with the first $d_\zeta$-vectors spanning the tangential space $T_xM$. By orthogonality, the remaining eigenvectors are in $N_xM$ so that the restricted mapping $Df(x)\big|_{N_xM}$ maps $N_xM$ into $N_xM$. As consequence  of~(\ref{eq:F-attractor2}), the restricted mapping $Df(x)\big|_{N_xM}:N_xM\to N_xM$ is injective and thus one-to-one.
	\end{rem}

	\black

	Further we introduce a notion of regularity that entails error estimates for certain Taylor approximations in our proofs. We will express our assumptions on the vector field $f$ and a certain local parametrisation of the manifold $M$ in this notion.

	\begin{defi} Let $U\subset \R^d$ be an open  set, $g:U\to \R^d$ be a mapping and $\alpha_g\in(0,1]$.
		\begin{enumerate}\item  We say that $g$ has \emph{regularity $\alpha_g$} if $g$ is continuously differentiable on $U$ with $\alpha_g$-H\"older continuous differential  $Dg$.
			\item Let, additionally, $M\subset \R^d$. We say that  $g:U\to \R^d$ has
			\emph{regularity $\alpha_g$ around $M$} if 
			\begin{itemize} \item $g$ is continuously differentiable on $M\cap U$ with $\alpha_g$-H\"older continuous differential and
				\item  there exists  a constant~$C$ such that for all $x\in M\cap U$ and $y\in U$ 
				$$|g(y) - (g(x)+Dg(x) (y-x)) |\leq C |y-x|^{1+\alpha_g}.$$
			\end{itemize}
		\end{enumerate}
	\end{defi}

	%
	
	We introduce certain kind of parametrisations of the manifold that will appear in our proofs.
	
	\begin{defi}\label{def:Phi}
		Let $(F,M)$ be an approximation problem and $\alpha_f, \alpha_\Phi,\alpha_\Psi\in(0,1]$. \smallskip

		1) 	 Let $U\subset \R^d$ be an open set intersecting $M$. A $C^1$-diffeomorphism  $\Phi:U_\Phi\to U$ is called \emph{nice representation for $M$ on $U$}  if the following is true:
		\begin{enumerate} \item[(A)] $U_\Phi$ is a convex subset of $\R^d$ such that for $(\zeta,\theta)\in \R^{d_\zeta} \times \R^{d_\theta}$
		$$
		(\zeta,\theta)\in U_\Phi \ \Rightarrow \ (\zeta,0) \in U_\Phi
		$$
		and $\Phi(U_\Phi \cap (\R^{d_\zeta} \times \{0\}^{d_\theta}))= U \cap M$.
		\item[(B)] There exists a family  $(P_{x}:x \in M\cap U)$ of isometric isomorphisms
			$P_{x}: \R^{d_\theta} \to  N_{x} M$ such that
 for every $(\zeta,\theta)\in U_\Phi\subset \R^{d_\zeta}\times \R^{d_\theta}$
			\begin{align}\label{rep:prop}
			\Phi(\zeta,\theta)= \Phi(\zeta,0) +  P_{\Phi(\zeta,0) }(\theta).
			\end{align}
			\end{enumerate}

		2) We say that $(F,M)$ has \emph{regularity $(\alpha_f,\alpha_\Phi,\alpha_\Psi)$} if for every $x\in M$ there exists a nice representation $\Phi:U_\Phi\to U$ of $M$ on a neighbourhood $U$ of $x$ such that  
		\begin{enumerate} \item[(a)] the vector field $f\big|_U= DF\big|_U$ has regularity $\alpha_f$ around $M$,
		\item[(b)] the mapping $\Phi$ has regularity $\alpha_\Phi$ around $\R^{d_\zeta}\times \{0\}^{d_\theta}$ and 
		\item[(c)] its inverse $\Psi:U\to U_\Phi$ has regularity $\alpha_\Psi$. 
		\end{enumerate}
		Further, an open set $U$ satisfying all the assumptions above are called nice representation for $M$ on $U$ with regularity $(\alpha_f,\alpha_\Phi, \alpha_\Psi)$.
	\end{defi}

		It is natural to ask for simple criteria to decide whether an approximation problem has a certain regularity. We discuss this issue in the following remark.
		
		\begin{rem} 
			\begin{enumerate}
				\item Let $\Psi: U \to V$ be a $C^1$-diffeomorphism with regularity $\alpha\in(0,1]$ and let  $U'\subset \R^d$ be a bounded and connected open set with $\overline{U'} \subset U$. By Theorem~1.3.4 of \cite{fiorenza2016holder}, it follows that the inverse $\Psi^{-1}\big|_{\Psi(U')}: \Psi(U') \to U'$ has also regularity $\alpha$. Hence, an approximation problem has regularity $(\alpha_f,\alpha,\alpha)$ if for every $x\in M$ there exists  a nice representation $\Phi:U_\Phi\to U$ of $M$ on a neighbourhood $U$ of $x$ such that  
						\begin{enumerate} \item[(a)] the vector field $f\big|_U= DF\big|_U$ has regularity $\alpha_f$ around $M$,
		\item[(b')] one of the mappings $\Phi$ or $\Psi$ has regularity $\alpha$. 
				\end{enumerate}
				\item Let $(F,M)$ be an approximation problem, so that $M$ is a $C^3$-manifold. Section~\ref{sec:Paralleltransport} shows that for every $x \in M$ there exist a neighbourhood $U \subset \R^d$ of $x$ and a nice representation $\Phi:U_\Phi \to U \in C^2$. Thus, after shrinking $U_\Phi$ we can guarantee that $D\Phi$ is Lipschitz and, again with Theorem~1.3.4 of \cite{fiorenza2016holder}, $\Phi$ is invertible with the differential of its inverse being a Lipschitz function. Hence, an approximation problem has regularity $(\alpha_f,1,1)$ if for every $x \in M$ there exists a neighbourhood $U$ of x such that
				\begin{enumerate}
					\item  $f\big|_U$ has regularity $\alpha_f$ around $M$ and
					\item[(b'')] $M$ is a $C^3$-manifold.
					
				\end{enumerate}
			\end{enumerate}
		\end{rem}

	Now we are able to state the main results.

	\begin{theorem}
		\label{thm:MainCLT_special}
		Let  $(F,M)$ be an approximation problem and  suppose that $(X_n)_{n\in\N_0}$ is the Robbins-Monro system and $(\bar X_n)$ the Ruppert-Polyak average as introduced in~(\ref{eq:dynsys1}) and~(\ref{eq:PR1}) with $(D_n)$,  $(\gamma_n)$, $(b_n)$, $(\bar b_n)$ and $(n_0(n))$ as in the introduction. 
		Furthermore, let $\Mconv$ denote the event that $(X_n)$ converges to an element of $M$ and denote by $X_\infty$ its limit which is a well-defined and measurable function on $\Mconv$. We consider the following assumptions:
		%
		%
		%
		\begin{enumerate}
			\item[(A.1)] Regularity.   $(F,M)$  has regularity $(\alpha_f, \alpha_\Phi, \alpha_\Psi)$, where $\alpha_f,\alpha_\Phi\in(0,1]$ and $\alpha_\Psi\in (\frac 12,1]$.
			\item[(A.2)] Assumptions on $(\gamma_n)$ and $(b_n)$.
			Set $\alpha=\alpha_\Psi\wedge\alpha_f\wedge \alpha_\Phi$ and  $\alpha'=\alpha_\Psi\wedge \frac {1+\alpha}2>\frac 12 $. Suppose that
			\begin{align}\label{assu:987}
			\Bigl(1-\frac\alpha{1+2\alpha}\Bigr)\vee  \Bigl(1- \frac 12\frac{\alpha_\Phi}{1+\alpha_\Phi}\Bigr) \vee \frac 1{2\alpha'} <\gamma<1 \text{ \ and \ }1+\rho>\gamma\alpha',
			\end{align}
			and set 
			$$
			\gamma_n=C_\gamma n^{-\gamma}  \text{ \ and \ } b_n=n^\rho. 
			$$
			\item[(A.3)] Assumptions on $(n_0(n))$. $(n_0(n))_{n\in\N}$ is a $\N_0$-valued sequence  with $0\leq n_0(n) <n$ for all $n\in\N$ that satisfies
			\begin{align}\label{assu:986}
			n_0(n)=o(n)\text{ \  and \ }	n_0(n)^{-1} = o \Bigl(n^{-\frac 1{2\gamma-1}\frac 1{1+\alpha_\Phi}} \wedge  n^{-\frac1\alpha \frac {1-\gamma}{2\gamma-1} } \Bigr).
			\end{align}
			 If $\rho<\gamma-1  $ we, additionally, assume that
			\begin{align}	 \label{assu:984}
				n_0(n)^{-1} = o \Bigl( n^{-\frac{\frac{1}{1+\alpha_\Phi} -(1+\rho)}{\gamma-(1+\rho)}} \Bigr).
			\end{align} 
			\item[(A.4)] 	Assumptions on $D_n$. For every $x \in M$ there exists an open neighbourhood $U\subset \R^d$ of $x$ and $n'\in\N$  so that
			$(\1_U(X_{n-1})D_n)_{n \in \N}$ is a sequence of uniformly $L^2$-integrable  martingale differences.
			Moreover,
			$$
			\lim_{n \to \infty} \cov(D_n|\cF_{n-1})= \Gamma, \text{ almost surely, on $\Mconv$.}
			$$
		\end{enumerate}
		%
		
		\black	
		
		Under the above assumptions the following is true:
		\begin{enumerate}\item \emph{CLT for the coefficients.}
			On $\Mconv$, one has 
			\begin{align}\label{state2a}
			\sqrt n \ (\bar X_n-\bar X_n^*)  \stackrel{\mathrm{stably}}\Longrightarrow  \frac{\rho+1}{\sqrt{2\rho + 1}} \bigl (Df(X_\infty)\big|_{N_{X_\infty} M}\bigr)^{-1} \ \Pi_{N_{X_\infty}M} \ \cN (0, \Gamma), 
			\end{align}
			where the right hand side stands for the random distribution  being obtained when applying the $\cF_\infty$-measurable linear transform $\bigl (Df(X_\infty)\big|_{N_{X_\infty} M}\bigr)^{-1} \ \Pi_{N_{X_\infty}M}$ onto a normally distributed random variable $\cN(0, \Gamma)$ with mean zero and covariance $\Gamma$.
			\item  \emph{CLT for the $F$-performance.}
			On $\Mconv$, one has
			\begin{align}\label{state3a}
			2 n ( F(X_\infty)-F(\bar X_n) ) \stackrel{\mathrm{stably}}\Longrightarrow \Bigl|\frac{\rho+1}{\sqrt{2\rho + 1}}\,\bigl (Df(X_\infty)\big|_{N_{X_\infty} M}\bigr)^{-1/2} \ \Pi_{N_{X_\infty}M} \ \cN(0, \Gamma)\Bigr|^2 ,
			\end{align}
			where   the right hand side stands for the random distribution  being obtained when applying the respective $\cF_\infty$-measurable operations    onto a normally distributed random variable with mean zero and covariance $\Gamma$.
		\end{enumerate}
		
		If assumption (A.1)  is true, there are feasible choices for $\gamma$ and $\rho$ that satisfy~(\ref{assu:987}) and for every such choice there exist feasible choices for $(n_0(n))_{n\in\N}$ satisfying~(A.3).
	\end{theorem}

	Theorem~\ref{thm:MainCLT_special} is a special case of Theorem~\ref{thm:MainCLT} below. 

	\begin{rem}\begin{enumerate}\item It is straight-forward to verify that the factor $ \frac{\rho+1}{\sqrt{2\rho + 1}}$ appearing on the right hand side of~(\ref{state2a}) and~(\ref{state3a}) is minimal for $\rho=0$. Furthermore, irrespective of the choice of allowed parameters we always have  $1>\gamma\alpha'$ so that $\rho=0$ is always a feasible choice, see~(\ref{assu:987}). Thus, taking a Ces\`aro average is always  optimal.
				\item
				The choice of $\alpha$'s that leads to the least restrictions  on the choice of $\gamma$ are $\alpha_\Phi=1$, $\alpha_\Psi=\frac 23$, $\alpha_f=\frac 12$. In that case all terms on the left hand side of the $\gamma$-condition~(\ref{assu:987}) equal~$\frac 34$ so that we are allowed to choose $\gamma$ in $(\frac 34,1)$. 
		\end{enumerate}
		%
	\end{rem}

	\begin{theorem} \label{thm:MainCLT} 
		Let  $(F,M)$ be an approximation problem and  suppose that $(X_n)_{n\in\N_0}$ is the Robbins-Monro system and $(\bar X_n)$ the Ruppert-Polyak average as introduced in~(\ref{eq:dynsys1}) and~(\ref{eq:PR1}) with $(D_n)$,  $(\gamma_n)$, $(b_n)$, $(\bar b_n)$ and $(n_0(n))$ as in the introduction. Let $(\sigma_n^\mathrm{RM})$ and $ (\delta_n^\mathrm{diff})$ be sequences of strictly positive reals and set 
		$$
		\sigma_n=\frac 1{\bar b_n}\sqrt {\sum_{l=n_0(n)+1}^n (b_l\delta_l^\mathrm{diff})^2}.
		$$
		Furthermore, let $\Mconv$ denote the event that $(X_n)$ converges to an element of $M$ and denote by $X_\infty$ its limit which is a well-defined and measurable function on $\Mconv$. We consider the following assumptions:
		\begin{enumerate}
			\item[(B.1)] Regularity.   $(F,M)$  has regularity $(\alpha_f, \alpha_\Phi, \alpha_\Psi)$, where $\alpha_f,\alpha_\Phi, \alpha_\Psi \in(0,1]$.
			\item[(B.2)] \emph{Technical assumptions on the parameters.}	
			Suppose that $(\gamma_n)$ is a monotonically decreasing sequence and
			$$
			n\gamma_n\to\infty , \ \ \gamma_n \to 0,
			$$
			\begin{align} \label{assu:123} \frac{b_{n+1}\gamma_n}{b_n\gamma_{n+1}}= 1+o(\gamma_n) , \ \ \limsup_{n\to\infty} \frac{1}{\gamma_n}\,\frac{\sigma^\mathrm{RM}_{n-1} - \sigma_n^\mathrm{RM}}{\sigma_n^\mathrm{RM}} =0 , \ \  \sigma^\mathrm{RM}_{n-1}\approx \sigma_n^\mathrm{RM}, 
			\end{align}
			and  for all sequences $(L(n))_{n\in\N}$ with $L(n)\leq n$ and $n-L(n)=o(n)$ one has
			$$
			\lim_{n\to\infty}\frac{\sum_{k=L(n)+1}^n (b_k \delta^{\mathrm{diff}}_k)^2}{\sum_{k=n_0(n)+1}^n (b_k \delta^{\mathrm{diff}}_k)^2}= 0.
			$$
			\item[(B.3)] Assumptions on $(n_0(n))$. $(n_0(n))_{n\in\N}$ is a $\N_0$-valued sequence  with $0\leq n_0(n) <n$ for all $n\in\N$ that satisfies $n_0(n)=o(n)$.
			\item[(B.4)] \emph{Assumptions on $D_n$.}  For every $x \in M$, there exist an open neighbourhood $U \subset \R^d$ of~$x$ so that $(\1_{U}(X_{n-1}) D_n)_{n\in\N}$ is a sequence of square integrable, martingale differences satisfying for all $\eps>0$, on $\Mconv$,
			$$
			\lim_{n\to\infty}(\delta_n^{\mathrm{diff}})^{-2} \cov(D_n|\cF_{n-1})=\Gamma \text{, \ almost surely,}
			$$
			\begin{align} \label{assu:D_n}
			\lim_{n\to\infty} ( \sigma_n )^{-2} \sum_{m=n_0(n)+1}^n \frac {b_m^2}{\bar b_n^2} \E[\1_{\{|D_m|>\eps\bar b_n  \sigma_n /b_m\}}|D_m|^2|\cF_{m-1}]=0, \text{ \ in probability,}
			\end{align}
			and
			\begin{align} \label{assu:D_n2}
			\limsup_{n\to\infty} \bigl( \frac{\sigma^\mathrm{RM}_n}{\sqrt{\gamma_n}}\bigr)^{-1} \E[\1_{U}(X_{n-1}) |D_n|^2]^{1/2}<\infty.
			\end{align}
			\item[(B.5)] \emph{Technical assumptions to control the error terms.} One has, as $n\to\infty$,
			\begin{align}\label{assu_step4}
			\frac {b_{n_0(n)}}{\bar b_n \gamma_{n_0(n)}} \sigma^\mathrm{RM}_{n_0(n)}=o(\sigma_n),
			\end{align}
			\begin{equation}\label{assu:eps2}
			(\eps_n^{(\ref{prop:zeta_est})})^{1+\alpha_\Phi}= o(\sigma_n)\text{ \ \ for \  \ }\eps_n ^{(\ref{prop:zeta_est})} :=\sum_{k=n_0(n)+1}^n \bigl((\sqrt{\gamma_k}\sigma^{\mathrm{RM}}_k)^{1+\alpha_\Psi}+\gamma_k(\sigma^{\mathrm{RM}}_{k-1})^{1+\alpha}\bigr)+\sqrt{\sum_{k=n_0(n)+1}^n \gamma_k  {(\sigma^{\mathrm{RM}}_k)^2}}
			\end{equation}
			\begin{equation} \label{assu:eps1}
			\eps_n^{(\ref{prop:tec_2})} := \frac 1{\bar b_n} \sum_{k=n_0(n)+1}^n b_k \bigl(\gamma_k^{-\frac {1-\alpha_\Psi}2} (\sigma^{\mathrm{RM}}_k)^{1+\alpha_\Psi}+(\sigma^{\mathrm{RM}}_{k-1})^{1+\alpha}+\sigma^{\mathrm{RM}}_{k-1}(\eps_n^{(\ref{prop:zeta_est})})^\alpha\bigr) =o(\sigma_n)
			\end{equation}and
			\begin{align}\label{assu:87456}
			\Bigl(\frac1{\bar b_n} \sum_{m=n_0(n){ +1}}^n b_m (\sigma^\mathrm{RM}_m)^2\Bigr)^{(1+\alpha_\Phi)/2}= o(\sigma_n).
			\end{align}
		\end{enumerate}
		Under the above assumptions the following is true:
		\begin{enumerate}\item \emph{CLT for the coefficients.}
			On $\Mconv$, one has 
			\begin{align}\label{state2}
			\sigma_n^{-1} \ (\bar X_n-\bar X_n^*)  \stackrel{\mathrm{stably}}\Longrightarrow \bigl (Df(X_\infty)\big|_{N_{X_\infty} M}\bigr)^{-1} \ \Pi_{N_{X_\infty}M} \ \cN(0, \Gamma), 
			\end{align}
			where the right hand side stands for the random distribution  being obtained when applying the $\cF_\infty$-measurable transform $\bigl (Df(X_\infty)\big|_{N_{X_\infty} M}\bigr)^{-1} \ \Pi_{N_{X_\infty}M}$ onto a normally distributed random variable with mean zero and covariance $\Gamma$.
			\item  \emph{CLT for the $F$-performance.}
			On $\Mconv$, one has
			\begin{align}\label{state3}
			2 \sigma_n^{-2} ( F(X_\infty)-F(\bar X_n) ) \stackrel{\mathrm{stably}}\Longrightarrow \bigl|\bigl (Df(X_\infty)\big|_{N_{X_\infty} M}\bigr)^{-1/2} \ \Pi_{N_{X_\infty}M} \ \cN(0, \Gamma)\bigr|^2 ,
			\end{align}
			where   the right hand side stands for the random distribution  being obtained when applying the respective $\cF_\infty$-measurable operations    onto a normally distributed random variable with mean zero and covariance $\Gamma$.
		\end{enumerate}
	\end{theorem}

	\begin{rem}
	If we, additionally, assume  in the theorem that there exists $L >0$, so that for every $x \in M$, the differential $Df(x)$ satisfies, for every $v \in N_xM$,
		$$
			\langle v, Df(x), v\rangle \leq -L |v|^2,
		$$
		then assumption~(\ref{assu:123}) can be relaxed to
		$$
			\frac{b_{n+1}\gamma_n}{b_n\gamma_{n+1}}= 1+o(\gamma_n) , \ \ \limsup_{n\to\infty} \frac{1}{\gamma_n}\,\frac{\sigma^\mathrm{RM}_{n-1} - \sigma_n^\mathrm{RM}}{\sigma_n^\mathrm{RM}} < L, \ \  \sigma^\mathrm{RM}_{n-1}\approx \sigma_n^\mathrm{RM}.
		$$
		
	\end{rem}
	
\black
\section{Geometric preliminaries} \label{sec:geometry}
	In this section we discuss some geometric properties of the $d_\zeta$-dimensional stable manifold~$M$. First, we derive that for an approximation problem $(F,M)$ in  sufficiently small neighbourhoods of $M$ the strength of attraction is uniformly bounded away from zero. Afterwards we discuss the well-definedness and regularity of the projection that maps every point to its nearest neighbour in $M$. 
	\begin{defi}
		Let $(F,M)$ be an approximation problem. We call an open  and bounded set $U\subset \R^d$ intersecting $M$   \emph{$(F,M)$-attractor} with stability $L$ and bound~$C$, for $C\ge L > 0$, if
		\begin{enumerate}
			\item  $\bar M\cap U=M\cap U$  and 
			\item  for every $x\in M\cap U$ and $v\in N_xM$
			\begin{align} \label{eq:F-attractor}
				-C |v|^2\leq \langle v, Df(x) v\rangle \leq - L |v|^2.
			\end{align}
		\end{enumerate}
	\end{defi}

\begin{lemma} \label{lem:existattractor}
Let $(F,M)$ be an approximation problem and $x\in M$, then $x$ admits an open  neighbourhood $U$ and constants $C,L>0$ such that  $U$ is an $(F,M)$-attractor with stability $L$ and bound~$C$.
\end{lemma}
\begin{proof}
Let $\Psi:U\to U_\Phi$ be a $C^1$-diffeomorphism  with $U$ being an open neighbourhood of $x$ and $U_\Phi\subset \R^d$ such that $\Psi(U \cap M)= U_\Phi\cap (\R^{d_\zeta} \times \{0\}^{d_\theta})$. 

First, we show that $\bar M\cap U=M\cap U$. Let  $z\in \bar M\cap U$. Then there exists a $M\cap U$-valued sequence $(z_n)$ with $z_n\to z$. Thus $\Psi(z_n)=\begin{pmatrix} \Psi_\zeta(z_n)\\0\end{pmatrix} \to \Psi(z)$, with $\Psi_\zeta(x_n)=(\Psi_1(x_n),\dots, \Psi_{d_{\zeta}}(x_n))$. Consequently, $\Psi_i(z)=0$ for all $i >d_\zeta$ and, hence,  $z\in M$. 

Second we show that for every bounded set $U'\subset U$ with $\overline {U'}\subset U$ there exist $C,L>0$ such that for all
 $z\in M\cap \overline{U'}$ and $v\in N_xM$
			\begin{align*} 
				-C |v|^2\leq \langle v, Df(z) v\rangle \leq - L |v|^2.
\end{align*}
It suffices to show that
	$$
	\cC:=\{(z,v) \in \R^d \times \R^d: z \in M \cap \overline{U'}, v\in N_zM, |v|=1\}
	$$
	is a compact set since then $C$ and $L$ can be chosen as
	$$
	-C=\min_{(z,v)\in \cC}  \langle v, Df(z) v\rangle \text{ and } -L=\max_{(z,v)\in \cC}  \langle v, Df(z) v\rangle
	$$
    with the minimum and maximum both being obtained and being in $(-\infty,0)$.
    Since $\cC$ is bounded it remains to prove closedness. Let $(z_n,v_n)_{n\in\N}$ be a $\cC$-valued sequence that converges to $(z,v)$. Since $M\cap \overline {U'}= \overline M \cap \overline{U'}$ is compact we have that $z\in M \cap \overline {U'}$. We denote by $\Phi$ the inverse of $\Psi$ and note that for all vectors $w\in \R^{d_\zeta}\times \{0\}$, $\partial _w \Phi(z_n)$ is in $T_{z_n}M$  which is perpendicular to $v_n\in N_{z_n}M$. Hence,
	$$
	0= \langle \partial _w \Phi(z_n), v_n\rangle \to \langle \partial _w \Phi(z), v\rangle
	$$
	and $v\perp  \partial _w \Phi(z)$. Since the considered vectors $\partial _w \Phi(z)$ span the tangent space $T_zM$ it follows that  $v\in (T_z M)^\perp = N_zM$ and we are done.    
\end{proof}

	\begin{rem} \label{rem:7346}
	Let  $(F,M)$ be an approximation problem. Then for $x\in M$  equation~(\ref{eq:F-attractor}) is satisfied for all $v\in N_x M$ if the spectrum of $Df(x)$ restricted to $N_xM$ is contained in $[-C,-L]$. Indeed, there is always an orthonormal basis of eigenvectors $v_1,\dots,v_d$ with $v_1,\dots,v_{d_\zeta}$ spanning $T_xM$ and $v_{d_\zeta+1},\dots,v_d$ spanning $N_x M$ and the equivalence follows by elementary linear algebra. 
	\end{rem}

	The remark entails the following corollary.  

\begin{cor} \label{lem:NormInequality}
	Let $U$ be a $(F,M)$-attractor with stability $L$ and bound $C$, $x \in U\cap M$ and $v\in N_x M$.
	Then for every $\gamma \in[0, C^{-1}]$ one has 
	$$
	|v+\gamma Df(x)v| \le (1-\gamma L)|v|.
	$$
\end{cor}

\begin{proof}
By Remark~\ref{rem:7346}, the spectrum of the restricted mapping $Df(x)\big|_{N_xM}:N_xM\to N_xM$ is contained in $[-C,-L]$. Hence, the spectrum of the restricted mapping $(\id+\gamma Df(x))\big|_{N_xM}$ is contained in $[1-\gamma C,1-\gamma L]\subset [0,1-\gamma L]$ which immediately implies the result since the latter mapping is diagonalizable.
\end{proof}

\black
For the next proposition we need the additional assumption, that the error of the first-order Taylor expansion of $f$ is locally uniform. If $f$ has regularity $\alpha_f$ around $M$ for some $\alpha_f \in (0,1]$, this follows immediately.

\begin{prop}\label{prop:dist_est}
	Let $U\subset \R^d$ be an $(F,M)$-attractor with stability $L$ and bound $C$. Suppose that for  $x\in U$ and $x' \in U\cap M$ 
	$$
		f(x)=Df(x')(x-x')+o(|x-x'|)\text{ as } |x-x'|\to 0
	$$
	with the small $o$ term being uniform in the choice of $x$ and $x'$.
	Then for every $L'\in(0,L)$ and $\delta >0$ there exists  $\rho>0$ such that for
	\begin{align} \label{def:U_delta^rho}
		U_\delta^\rho:= \bigcup_{\substack{y\in M:\\d(y,U^c)>\delta}} B_\rho (y)
	\end{align}
	one has for all $x\in U_\delta^\rho$ and $\gamma\in[0,C^{-1}]$
	\begin{align}\label{eq190206-1} d(x+\gamma f(x), M)\leq (1-\gamma L') d(x,M).
	\end{align}
\end{prop}
\begin{proof}
	Choose $\rho\in (0,\frac{1}{2}\delta]$ such that for all $x,x'\in U$ with $x'\in M$ and $|x'-x|\leq \rho$
	$$
	|f(x)-Df(x') (x-x')| \leq (L-L') |x-x'|.
	$$
	Let $x \in U_\delta^\rho$. Then by definition of $U_\delta^\rho$ there exists $x'\in M$ with $d(x,x')<\rho$ and $d(x',U^c)>\delta$. We denote by $z\in \bar M$ an element with
	$$
	d(x,z)=d(x,\bar M)=d(x,M)<\rho.
	$$
	Note that $d(x',z)\leq d(x',x)+d(x,z)<2\rho\leq \delta$ so that  $z\in B_\delta(x')\subset U$ and, hence, $z\in \bar M\cap U= M\cap U$.
	Take $v\in T_{z}M$ and a $C^1$-curve $\gamma:(-1,1)\to M$ with $\gamma(0)=z$ and $\dot \gamma(0)=v$. Then since $t\mapsto d(\gamma(t),x)^2$ has a minimum in $0$ we get that
	$$
	0=	\frac{\mathrm{d}}{\mathrm{d} t} \ d(\gamma(t),x)^2\Big|_{t=0}= 2 \langle z-x,v \rangle.
	$$
	Thus $x-z\in N_{z}M$.
	With Lemma~\ref{lem:NormInequality}  we obtain that  for $\gamma \in [0,C^{-1}]$
	\begin{align*}
		d(x+\gamma f(x), M)&\leq d(x+\gamma f(x), z) \leq  |(\mathrm{Id}+\gamma  Df(z))(x-z)|+ \gamma |f(x) -Df(z)(x-z)|\\
		&\leq (1-\gamma L) |x-z|+ \gamma (L-L') |x-z|= (1-\gamma L') d(x, M).
	\end{align*}
	%
	%
\end{proof}

	We consider the projection onto $M$ which is defined as follows. For $x \in \R^d$ we set
	$$
	x^*= \mathrm{argmin}_{y\in M} d(x,y),
	$$
	if there is a unique minimizer.
	
	We will show that for a nice representation $\Phi:U_\Phi\to U$ of $M$ on some open and bounded set $U$ (in the sense of Definition~\ref{def:Phi}) and its inverse $\Psi$ we have
	$$
	x^*=\Phi(\Psi_\zeta(x),0)
	$$
	for all $x\in U$ that are sufficiently close to $M$. Here $\Psi_\zeta$ represents the first $d_\zeta$ coordinates of $\Psi$, that is
	$\Psi_\zeta(x)=(\Psi_1(x),\dots, \Psi_{d_\zeta}(x))$ for $x\in U$. 
%
%
%
%

	\begin{lemma} \label{lem:Geometry} Let $\delta>0$ and $U\subset \R^d$ an open and bounded set and $\Phi:U_\Phi\to U$ a nice representation for $M$ on $U$.
		\begin{enumerate}
		\item[(i)]
			There exists $\rho\in(0,\delta/4]$ such that for every $x \in M$ with $d(x,U^c)>\delta/2$ and $\theta \in \R^{d_\theta}$ with $|\theta|<\rho$ it holds
			\begin{align}\label{as:89456}
				\Psi(x) +  \begin{pmatrix} 0 \\ \theta\end{pmatrix} =  \begin{pmatrix} \Psi_\zeta(x)  \\ \theta\end{pmatrix} \in U_\Phi.
			\end{align}
			\item[(ii)] Suppose that  $\rho>0$ is as in (i). Then, for every $x\in  U_\delta^\rho$,  
			$x^*$ is well-defined and one has the following:
			\begin{itemize}\item  $x^*=\Phi(\Psi_\zeta(x),0)$,
			\item  the segment connecting $x$ and $x^*$ lies in $U$  and 
			\item $|\Psi_\theta(x)| = d(x,M) $ \ and \ $d(x^*, U^c)>\delta/2$.
			\end{itemize}
		\end{enumerate}
	\end{lemma}
	\begin{proof}
			{\bf (i):} Let $\delta>0$ and note that
		$$
		M':= \{x\in  \bar M: d(x,U^c)\geq \delta/2\}\subset \bar M\cap U=M\cap U
		$$
		is a compact set.
		Hence, the continuous mapping 
		$$
		M'\ni x \mapsto d(\Psi(x),U_\Phi^c)
		$$
		attains its minimum, say $\rho'$, which is strictly positive since $\Psi$ does not attain values in the closed set $U_\Phi^c$. Obviously, property (i) holds for $\rho=\min(\rho', \delta/4)$.		
%
%
%
%
		
		{\bf (ii):} Let $\rho\in(0,\delta/4]$ as in (i) and let $x \in U_\delta^\rho$. First we show that an element $z\in \bar M$ with
		$$
		d(x,z)=d(x,\bar M)=d(x,M)
		$$
		lies in $M\cap U$ and satisfies $x-z\in N_z M$. By definition of $U_\delta^\rho$ there exists $x'\in M$ with $d(x,x')<\rho$ and $d(x',U^c)>\delta$. Thus $d(x',z)\leq d(x',x)+ d(x,z)<2\rho$ and $d(z,U^c)\geq d(x',U^c)-d(x',z)>\delta-2\rho\geq \delta/2$ so that $z\in U$ and hence also $z\in \bar M\cap U=M \cap U$.
		
		Take $v\in T_{z}M$ and a $C^1$-curve $\gamma:(-1,1)\to M\cap U$ with $\gamma(0)=z$ and $\dot \gamma(0)=v$. Then since $t\mapsto d(\gamma(t),x)^2$ has a minimum in $0$ we get that
		$$
		0=	\frac{\mathrm{d}}{\mathrm{d} t} \ d(\gamma(t),x)^2\Big|_{t=0}= 2 \langle z-x,v \rangle.
		$$
		Thus $x-z\in N_{z}M$. 	
		We recall that $|x-z|=d(x,z)<\rho$  so that as consequence of the representation property~(\ref{rep:prop}) there exists $\theta\in \R^{d_\theta}$ with $|\theta|=|x-z|<\rho$ and
		$$
		x= z+ P_z(\theta).
		$$
		Moreover, recalling that $d(z,U^c)>\delta/2$ we get with (i) that $(\Psi_\zeta(z),\theta)$ is in $U_\Phi$ and hence		
		 $x = \Phi(\Psi_\zeta(z),\theta)$. 
		An application of $\Psi_\zeta$ yields that $\Psi_\zeta(x)=\Psi_\zeta(z)$ so that $$z=\Phi(\Psi_\zeta(z),0)=\Phi(\Psi_\zeta(x),0)$$
		is  the unique minimizer and $x^*=z$.		
		Furthermore, with (\ref{as:89456}) the segment connecting $x^*$ and $x$, which is $\gamma:[0,1] \to \R^d,$ $t \mapsto \Phi(\Psi_\zeta (x), t \theta)$ lies in $U$ and
		$$
		|\Psi_\theta(x)|= |\theta|= d(z,x)= d(x,M).
		$$
\end{proof}

\begin{prop}\label{le:feasible}
Let $(F,M)$ be an approximation problem with regularity $(\alpha_f,\alpha_\Phi,\alpha_\Psi)$ and suppose that all assumptions of Theorem~\ref{thm:MainCLT} are satisfied. We call a triple $(U,\delta,\rho)$ consisting of an open set $U\subset \R^d$ and $\delta,\rho>0$ \emph{feasible}, if
\begin{itemize}\item  there exists a nice  representation $\Phi: U_\Phi\to U$ for $M$ on $U$ with regularity $(\alpha_f,\alpha_\Phi,\alpha_\Psi)$, 
\item $U$ is an $(F,M)$-attractor with stability $L$ and bound $C$ for some values $L,C>0$,
\item $(\1_U(X_{n-1}) D_n)_{n\in\N}$ is a sequence of $L^2$-martingale differences satisfying (\ref{assu:D_n2}),
\item $\delta>0$ and $\rho\in (0,\delta/4]$ are such that (i) of Lemma~\ref{lem:Geometry} is true and inequality~(\ref{eq190206-1}) holds for a $L'\in(0, L)$.
\end{itemize}
Then there exists a countable set of feasible triples $(U,\delta,\rho)$ such that the respective subsets $U_\delta^\rho$ of $\R^d$ cover the manifold $M$. 
\end{prop}

\begin{proof}
For every $x\in \R^d$ and every feasible triple $(U,\delta,\rho)$ we denote by $$R_x(U,\delta,\rho)=\sup \{r\geq 0: B_x(r) \subset U_\delta^\rho\}
$$
the radius of the triple $(U,\delta,\rho)$ at $x$. Note that by definition for $x,y\in\R^d$,
$|R_x(U,\delta,\rho)-R_y(U,\delta,\rho)|\leq |x-y|$ so that the function
$$
\R^d\ni x\mapsto R_x = \sup\{R_x(U,\delta,\rho): (U,\delta,\rho)\text{ is feasible}\}
$$
is Lipschitz continuous with Lipschitz constant $1$. (Possibly, all function values are infinite.)

Now fix a  $\kappa>0$ and a countable set $\II_\kappa\subset \R^d$ such that 
$$\bigcup_{z\in \II_\kappa} B_z(\kappa/3)= \R^d.$$
We construct a collection $\cU_\kappa$ of feasible triples as follows. For every $z\in \II_\kappa$ with $R_z\geq 2\kappa/3$ we add a triple with $z$-radius greater or equal to $\kappa/2$. For every $z\in \II_\kappa$ with $R_z< 2\kappa/3$ we do not add a triple. 
Then $\cU_\kappa$ is countable and for every $x\in M$ with $R_x\geq \kappa$ there exists a $z\in \II_\kappa$ with $|x-z|\leq \kappa/3$. Hence $R_z\geq 2\kappa/3$ and we thus added a triple $(U,\delta,\rho)$ with $z$-radius greater or equal to $\kappa/2$ which obviously contains $x$. Consequently, $\cU_\kappa$ is a countable set of feasible triples that covers at least $\{x\in M: R_x\geq \kappa\}$. By a diagonalisation argument, we obtain a countable set 
$\bigcup_{n\in\N} \cU_{1/n}$ of feasible triples that covers $\bigcup_{n\in\N} \{x\in M: R_x\geq 1/n\}=M$.
\end{proof}

\begin{rem}\label{rem:feas84} We consider the setting of Theorem~\ref{thm:MainCLT}. 
Let $\cU$ be a countable set of feasible triples that covers $M$ as in Lemma~\ref{le:feasible}. For a feasible triple $(U,\delta,\rho)$ the set $U_\delta^\rho$ is open and we consider
 the event $\Uconv_{\delta,\rho}$ that $(X_n)$ converges to an element of $M\cap U_\delta^\rho$. 
Then the covering property of $\cU$ ensures that
$$
\Mconv= \bigcup_{(U,\delta,\rho)\in \cU} \Uconv_{\delta,\rho}
$$
and by Lemma~\ref{lem:stableonsubsets} the proof of Theorem~\ref{thm:MainCLT} is achieved  once we showed stable convergence on~$\Uconv_{\delta,\rho}$ for general feasible triples $(U,\delta,\rho)$.
\end{rem}

\section{$L^2$-error bounds} \label{sec:L2conv}
In this chapter, we control the behavior of the  Robbins-Monro scheme around an $(F,M)$-attractor at late times in terms of the distance to $M$ in the $L^2$-norm. We will later need these estimates to control errors that we infer when comparing the original dynamical system with a linearised one.

As in the chapters before, let $(F,M)$ be an approximation problem and let $U\subset \R^d$ be an $(F,M)$-attractor with stability $L$ and bound $C$. We denote by $f=DF$ the Jacobi matrix of $F$ and consider a dynamical system $(X_n)$ given by
\begin{equation} \label{eq:DynSysNewRn}
X_n=X_{n-1} +\gamma_n(f(X_{n-1})+\underbrace{R_n+D_n}_{=U_n})
\end{equation}
with
\begin{itemize}
	\item $X_0 \in \R^d$ is a fixed deterministic starting value,
	\item $R_n$ being $\cF_{n-1}$-measurable and
	\item $D_n$ is $\cF_n$-measurable and $(\1_{U}(X_{n-1})  D_n)_{n\in\N}$ is a sequence of square integrable martingale differences.
\end{itemize}
Thus, in this chapter we also allow the process to have a previsible bias which should be of lower order than the martingale noise. This assumption will be made precise in the following theorem. We obtain the process introduced in (\ref{eq:dynsys1}) by choosing $R_n\equiv 0$.

\black

\black

	\begin{theorem}\label{thm:L2_bound} Let $U\subset \R^d$ be an $(F,M)$-attractor with stability $L$ and bound $C$.
		Suppose that for  $x\in U$ and  $x' \in U\cap M$ 
		$$
		f(x)=Df(x')(x-x')+o(|x-x'|)\text{ \ as \ }|x-x'|\to0
		$$
		with the $o$-term being uniform in $x$ and $x'$.
		Let $(\gamma_n)_{n\in\N}$ and $(\sigma_n)_{n\in\N}$  sequences of strictly positive reals with $\lim_{n\to\infty}\gamma_n = 0$, $\sum_{n=1}^\infty \gamma_n=\infty$ and
		\begin{align} \label{eq23478}
		L'':= \limsup_{n\to\infty} \frac{1}{\gamma_n}\,\frac{\sigma_{n-1} - \sigma_{n}}{\sigma_{n}}  < L
		\end{align}
		and suppose that $(X_n)_{n\in\N_0}$ satisfies recursion~(\ref{eq:DynSysNewRn}).
			Let $\delta,\rho>0$  
		be such that Prop.~\ref{prop:dist_est} is true for a  $L' \in (L'',L)$, that is for all $x\in U_\delta^\rho$ and $\gamma\in[0,C^{-1}]$ one has $$d(x+\gamma f(x),M)\leq (1-\gamma L') d(x,M).$$
		Furthermore, assume that  
		\begin{align}\label{assu1_mom}
		\limsup_{n\to\infty} \sigma_n^{-1} \E\bigl[ \1{\{X_{n-1}\in U_\delta^\rho\}} |R_n |^2\bigr]^{1/2}<\infty
		\end{align}
		and
		\begin{align}\label{assu2_mom}
		\limsup_{n\to\infty}\Bigl(\frac{\sigma_n}{\sqrt{\gamma_n}}\Bigr)^{-1} \, \E[\1{\{X_{n-1}\in U_\delta^\rho\}} |D_n|^2]^{1/2}<\infty.
		\end{align}
	 Then there exist $C_{(\ref{thm:L2_bound})} \in (0,\infty)$  such that for all $N\in \N$,
		\begin{align} \label{equ:L2conv}
		\limsup_{n\to\infty} \sigma_n^{-1} \E\bigl[\1_{\{X_m\in U_\delta^\rho\text{ for } m=N,\dots,n-1\}} d(X_n,M)^2\bigr]^{1/2}<C_{(\ref{thm:L2_bound})}.
		\end{align}
	\end{theorem}

	\begin{proof} 
		Let $L'\in (L'',L)$ and $\delta,\rho>0$ as in the theorem. 
		By monotonicity it suffices to restrict attention to large $N$. For sufficiently large constants $C_1$ and $C_2$ we can fix $N_0\in\N$ such that for all $n\geq N_0$
		\begin{align}\label{eq9567}
		\gamma_n\leq C^{-1}, \   \E\bigl[ \1{\{X_{n-1}\in U_\delta^\rho\}} |R_n |^2\bigr]\leq C_1 \sigma_n^2 \text{ \ and \ } \E[\1{\{X_{n-1}\in U_\delta^\rho\}} |D_n|^2]\leq C_2 \frac{\sigma_n^2}{\gamma_n}.
		\end{align}
		Now fix $N\geq N_0$ and  consider
		$$
		\IU_n=\{\forall l=N,\dots,n: X_l \in U_\delta^\rho\}\text{, \  for $n\geq N$}.
		$$
		One has  for $n>N$	\begin{align}\begin{split}\label{eq9457}
		\E[\1_{\IU_{n-1}}& d(X_{n-1} +\gamma_n(f(X_{n-1})+R_n+D_n),M)^2]\\
		&\leq \underbrace{  \E[\1_{\IU_{n-1}} d(X_{n-1} +\gamma_n(f(X_{n-1})+R_n),M)^2]}_{=:I_1(n)}+\underbrace{  {\gamma_n^2}\, \E[\1_{\IU_{n-1}} |D_n|^2]}_{=:I_2(n)}.
		\end{split}	\end{align}
		Moreover, by~(\ref{eq190206-1})  one has on the event $\IU_{n-1}$ for arbitrary  $a>0$ 
		\begin{align*}
		d(X_{n-1} +\gamma_n(f(X_{n-1})+R_n),M)^2&\leq (1-L'\gamma_n)^2 d(X_{n-1}, M)^2+ 2 \gamma_n d(X_{n-1}, M) |R_n| + \gamma_n^2 |R_n|^2\\
		&\leq ((1-L'\gamma_n)^2+ a \gamma_n)\,  d(X_{n-1}, M)^2+ (\frac 1a \gamma_n+\gamma_n^2)  |R_n|^2.
		\end{align*}
		Consequently, with~(\ref{eq9567}) 
		$$
		I_1(n) \leq ((1-L'\gamma_n)^2+a\gamma_n) \, \E[ \1_{\IU_{n-1}} d(X_{n-1} ,M)^2]+ C_1 (\frac 1a \gamma_n+\gamma_n^2) \sigma_n^{2}.
		$$
		Now note that as $n\to\infty$, $(1-L'\gamma_n)^2= 1-2L'\gamma_n +o(\gamma_n)$. Moreover,
		$$
		\frac {\sigma_{n-1}}{\sigma_n} =1+\frac {\sigma_{n-1}-\sigma_n}{\sigma_n} \leq 1+ L''\gamma_n +o(\gamma_n)
		$$
		so that 
		$$
		\frac {\sigma^2_{n-1}}{\sigma^2_n} ((1-L'\gamma_n)^2+ a \gamma_n)\leq ( 1+2L''\gamma_n+o(\gamma_n)) (1-(2L'-a)\gamma_n+o(\gamma_n) ) = 1-(2L'-2L''-a) \gamma_n +o(\gamma_n).
		$$
		Recall that $L'>L''$ and we fix $a,b>0$ such that $2L'-2L''-a>b$. Then for sufficiently large $n\in\N$
		$$
		\frac {\sigma^2_{n-1}}{\sigma^2_n} ((1-L'\gamma_n)^2+ a \gamma_n) \leq 1-b \gamma_n
		$$
		and by increasing $N$ we can guarantee that the previous inequality holds for all $n>N$. Thus
		$$
		\sigma_n^{-2} I_1(n)\leq (1-b\gamma_n) \sigma_{n-1}^{-2} \E[ \1_{\IU_{n-1}} d(X_{n-1} ,M)^2]+ C_1 (\sfrac 1a +\sfrac 1C) \gamma_n.
		$$
		Additionally we get with~(\ref{eq9567}) that  $\sigma_n^{-2} I_2(n)\leq  C_2  \gamma_n$. 
		This implies that the expectation
		$$\varphi_n:= \sigma_n^{-2} \E[\1_{\IU_{n}}  d(x_{n},M)^2] \qquad( n\geq N)$$
		satisfies for $n>N$
		\begin{align*}
		\varphi_n &\leq \sigma_n^{-2} (I_1(n)+I_2(n)) \leq (1-b\gamma_n)  \varphi_{n-1}+ \underbrace{(C_1 (a^{-1}+C^{-1})+ C_2)}_{=:C_3 }  \gamma_n.
		\end{align*}
		It follows that
		$$
		\varphi_n-\frac {C_3}{b} \leq (1-b\gamma_n)\Bigl( \varphi_{n-1}-\frac {C_3}{b} \Bigr)
		$$
		and by iteration that 
		$$
		\varphi_n-\frac {C_3}{b} \leq \Bigl( \varphi_{N}-\frac {C_3}{b}\Bigr )\prod_{l=N+1}^n (1-b\gamma_l)\to0,
		$$where convergence follows since $\sum_{l=N+1}^\infty \gamma_l=\infty$. Therefore, 
		$$
		\limsup_{n\to\infty} \varphi_n\leq \frac {C_3}b.
		$$
		Note that the statement remains valid with the same constant on the right hand side when increasing $N$.
	\end{proof}

\section{The Ruppert-Polyak system for linear systems} \label{sec:linSys}

	In this section, we provide a central limit theorem for a particular linear system. It will be the main technical tool for proving Theorem~\ref{thm:MainCLT}. More explicitly, we will show that on the level of coordinate mappings the system is approximated up to lower terms by the system analysed here.
	
	Again $(\gamma_n)_{n \in \N}$ denotes a monotonically decreasing sequence of nonnegative reals which converges to $0$. Additionally,  $(n_0(n))_{n\in\N}$ is an increasing $\N_0$-valued sequence with $n_0(n)\leq n$ that tends to infinity and for each $n\in\N$, let $H_{n}$ be a $\cF_{n_0(n)}$-measurable matrix. 
	We set for $n,i,j\in\N$ with $i\leq j$
	\begin{align}\label{eq:def_H} \cH_n[i,j]= \prod_{r=i+1}^j (\1+\gamma_r H_n) \text{ \ and \ } \bar \cH_n[i,j]= \sum_{r=i}^j \frac{\gamma_i  b_r}{b_i}  \cH_n[i,r].
	\end{align}
	Based on a sequence  $(\cD_l)_{l \in \N}$  of $\R^d$-valued random variables we 
	consider the dynamical system $(\Xi_n)_{n\in\N}$ with
	\begin{align}
		\Xi_n:= \frac 1{\bar b_n}
		\sum_{i=n_0(n)+1}^n b_i \bar \cH_n[i,n] \,\cD_i
	\end{align}
	and
	$$
	\bar b_n =\sum_{i=n_0(n)+1}^n b_i.
	$$

	\begin{theorem} \label{thm:mainlin} 
		Let $A\in\cF_\infty$ and $(\delta_n)_{n\in\N}$ be a sequence of strictly positive reals. We assume the following assumptions:
		\begin{enumerate}\item \emph{Technical assumptions on the parameters.}
			$$n\gamma_n\to\infty , \ \  \frac{b_{n+1}\gamma_n}{b_n\gamma_{n+1}}= 1+o(\gamma_n)$$
			and  for all sequences $(L(n))_{n\in\N}$ with $n_0(n)\leq L(n)\leq n$ and $n-L(n)=o(n)$ one has
			$$
			\lim_{n\to\infty}\frac{\sum_{k=L(n)+1}^n (b_k \delta_k)^2}{\sum_{k=n_0(n)+1}^n (b_k \delta_k)^2}= 0.
			$$
			\item \emph{Assumptions on $H_n$.} $(H_n)_{n\in\N}$ is a sequence of symmetric matrices with each $H_n$ being $\cF_{n_0(n)}$-measurable and
			$$
			\lim_{n\to\infty} H_n =H\text{, \ almost surely, on $A$,}
			$$
			for a random symmetric matrix $H$ with $\max \sigma(H)<0$.
			\item \emph{Assumptions on $\cD_k$.}  $(\cD_k)_{k\in\N}$ is  a sequence of  square integrable martingale differences that satisfies for a random matrix $\Gamma$, on $A$,
			\begin{enumerate}
				\item ${\displaystyle
					\lim_{m \to\infty} \big\| \cov( \delta_m^{-1}  \cD_m|\cF_{m-1})-\Gamma\big\|=0 \text{, \ almost surely,  and}}$
				\item  for  ${\displaystyle
					\sigma_n=\frac 1{\bar b_n}\sqrt {\sum_{m=n_0(n)+1}^n (b_m\delta_m)^2}}$
				and all $\eps>0$, one has
				$$ \lim_{n\to\infty} \sigma_n^{-2} \sum_{m=n_0(n)+1}^n \frac {b_m^2}{\bar b_n^2} \E \bigl[\1_{\{|\cD_m|>\frac{\eps\bar b_n \sigma_n}{b_m}\}} |\cD_m|^2\big| \cF_{m-1}\bigr]=0, \text{ in probability.}
				$$
			\end{enumerate}
			Then it follows
			that
			$$
			\sigma_n ^{-1} \Xi_n \stackrel{\mathrm{stably}}\Longrightarrow H^{-1} \cN(0, \Gamma)\text{, \ on $A$,}
			$$
			where the right hand side stands for the random distribution being obtained when applying the $\cF_\infty$-measurable matrix $H^{-1}$ onto a normally distributed random variable with mean zero and covariance $\Gamma$.
		\end{enumerate}
		
	\end{theorem}

	The proof relies on two technical estimates taken from~\cite{Der19}. Based on a monotonically decreasing sequence  $(\gamma_n)_{n\in\N}$ of strictly positive reals we define times $(t_n)_{n\in\N_0}$ via
	$$
	t_n= \sum_{m=1}^n \gamma_m.
	$$
	
	We cite~\cite[Lemma 2.3]{Der19}.
	\begin{lemma}\label{le:count_set}
		If $\lim_{n\to \infty}n\gamma_n=\infty$, then   for every $C>0$
		$$
		\lim_{n\to\infty} \frac 1n \#\{l\in\{1,\dots,n\}: t_n-t_l\leq C\}=0.
		$$
	\end{lemma}
	
	We cite~\cite[Lemma~2.2]{Der19}.
	\begin{lemma}\label{le:dom_tech}
		We define for each $l \in \N$ the function $F_l:[0,\infty)\to[0,\infty)$ 
		by demanding that for every $k\geq l$ and $s\in  [t_{k-1}-t_{l-1},t_k-t_{l-1})$
		$$
		F_l(s) = \frac{\gamma_l b_{k}}{\gamma_{k}b_l}.
		$$
		If $\frac{b_{n+1}\gamma_n}{b_n \gamma_{n+1}}= 1+o(\gamma_n)$, then
		\begin{enumerate}
			\item [i)] $F_l$ converges pointwise to 1
			\item [ii)] there exists a measurable function $\bar F$ and $n_0\in\N$ such that $F_l \le \bar F$ for all $l \ge n_0$  and
			$$
			\int_{0}^{\infty} \bar F(s)(s \vee 1)e^{-Ls} ds < \infty.
			$$
		\end{enumerate}
	\end{lemma}
	
	The following lemma is a slight variation of~\cite[Lemma 2.6]{Der19}.
	\begin{lemma} \label{lem:Huniformly}
		Suppose that $\frac {b_{n+1} \gamma_n}{b_n\gamma_{n+1}}=1+o(\gamma_n)$.
		\begin{enumerate} \item Let  $(H_n)_{n\in\N}$ be a (deterministic) sequence of symmetric matrices
			that converges to a matrix $H$ with $\sigma(H)\subset(-\infty,0)$. Then  $\bar\cH_n$ as defined in
			~(\ref{eq:def_H}) satisfies 
			$$
			\limsup_{l,n\to\infty,\, t_n-t_l\to \infty} \bigl\|\bar \cH_{n}[l,n]+ H^{-1}\bigr\|= 0.
			$$
			\item Let  $L,C\in(0,\infty)$. There exist constants $C_{(\ref{lem:Huniformly})}<\infty$ and $N_{(\ref{lem:Huniformly})}\in\N$  such that for every symmetric matrix $H$ with 
			$$
			\sigma(H)\subset [-C,-L]
			$$
			one has for  every $l,n\in\N$ with $N_{(\ref{lem:Huniformly})}\leq l\leq  n$
			$$\|\bar \cH[l,n]\|\leq C_{(\ref{lem:Huniformly})}.
			$$
		\end{enumerate}
	\end{lemma}

	\begin{proof} [Proof of Lemma~\ref{lem:Huniformly}]
		(1) Let $l,k\in\N_0$ with $l\leq k$. We will first provide an estimate for $e^{(t_k-t_l)H_n} -\prod_{r=l+1}^k (\1+\gamma_r H_n)$ on the basis of the following telescoping sum representation:
		\begin{align}\label{eq9435763}
			e^{(t_k-t_l)H_n} -\prod_{r=l+1}^k (\1+\gamma_r H_n)= \sum_{q=l+1}^k e^{(t_{q-1}-t_l)H_n} (e^{\gamma_qH_n}-(\1+\gamma_q H_n)) \prod_{r=q+1}^k (\1+\gamma_r H_n).
		\end{align}
		Each term in the latter sum is a product of three matrices and we will analyse the norm of these individually.
		
		We will use that the spectrum of a matrix depends continuously on the matrix. Let $\lambda^{(1)},\dots, \lambda^{(d)}$ denote the eigenvalues of $H$. For $n\in\N$ one can enumerate the eigenvalues $\lambda_n^{(1)}, \dots, \lambda_n^{(d)}$ of $H_n$ in such a way that $\lim_{n\to\infty}\lambda_n^{(i)}= \lambda^{(i)}$ for every $i=1,\dots,d$ (see for instance \cite[VI.1.4]{Bha13}).
		By assumption,   $\sigma(H) \subset (-\infty,0)$ so that there exist $C,L > 0 , n_0 \in \N$ with
		$$
		\sigma(H_n) \subset [-C,-L] \; \text{ for all } n \ge n_0.
		$$
		Next note that for $\delta\geq 0$, $\1+\delta H_n$ has eigenvalues $1+\delta\lambda_n^{(1)}, \dots, 1+\delta\lambda_n^{(d)}$. These are all elements of the  interval $[1-\delta C,1-\delta L]$ and provided that $\delta\leq 1/C$ we get that the spectral radius and likewise the matrix norm of $\1+\delta H_n$ are bounded by $1-\delta L$. 
		By possibly increasing the value of~$n_0$ we can guarantee that for all $k \ge n_0$,  $\gamma_k \le \frac{1}{C}$. 
		For such $n_0$ we conclude that  for all $k\geq l \ge n_0$ and $n\geq n_0$,
		\begin{align}\label{eq874574}
			\| \cH_n[l,k] \|=\Bigl\| \prod_{r=l+1}^{k} (\1 + \gamma_r H_n)\Bigr\| \le \prod_{r=l+1}^k \underbrace{(1-\gamma_r L)}_{\leq e^{-\gamma_r L}} \le e^{-L(t_k-t_l)}.
		\end{align} 
		Moreover, $e^{(t_{k}-t_l)H_n}$ has eigenvalues $\exp\{(t_{k}-t_l) \lambda_n^{(1)}\},\dots,$ $\exp\{(t_{k}-t_l) \lambda_n^{(d)}\}$ so that
		$$\|e^{(t_{k}-t_l)H_n}\|\leq e^{-L(t_{k}-t_l)}.
		$$
		Recall further that for a $d\times d$-matrix $A$
		$$
		\|e^A- (\1+A)\|\leq \frac12 e^{\|A\|} \|A\|^2.
		$$
		Altogether, we thus get with~(\ref{eq9435763}) that
		\begin{align}\begin{split}\label{eq43568}
				\|e^{(t_k-t_l)H_n} -\cH_n[l,k]\|&=\Bigl\|e^{(t_k-t_l)H_n} -\prod_{r=l+1}^k (\1+\gamma_r H_n)\Bigr\|\\
				&\leq  \frac 12 e^{-(t_{k}-t_l)L+\gamma_1L} e^{\gamma_1\|H_n\|} \|H_n\|^2 \sum_{q=l+1}^k \gamma_q^2\\
				&\leq  C' e^{-(t_{k}-t_l)L}   \gamma_l (t_k-t_l) 
		\end{split}\end{align}
		with $C':=\sup_{n\geq n_0}  \|H_n\|^2 e^{\gamma_1(L+\|H_n\|)}\leq C^2e^{\gamma_1(L+C)} <\infty$.

		We note that, as $H_n$ are symmetric matrices with $\sigma(H_n)\subset [-C,-L]$ for all $n \ge n_0$, $H_n$ is invertible and
		$$
		H_n^{-1} \to H^{-1}.
		$$ 
		Therefore, it suffices to show that
		$$
		\limsup_{l,n \to\infty,\  t_n-t_l\to\infty}    \bigl\|\bar \cH_{n}[l,n]+ H_n^{-1}\bigr\|= 0.
		$$
		To establish this we consider for $n\geq l\geq n_0$,
		$I_1=I_1(l,n)=\bar \cH_n[l,n]$, 
		$$
		I_2= I_2(l,n)=\frac{\gamma_l}{b_l}\sum_{k=l}^n b_k e^{(t_k-t_l)H_n} \text{ \ and \ } I_3=I_3(l,n)=\sum_{k=l}^n \gamma_k e^{(t_k-t_l)H_n}$$
		and omit the $(l,n)$-dependence in the notation.
		
		We analyse $\|I_1-I_2\|$. 
		Using $F_l$ as introduced in Lemma~\ref{le:dom_tech} we get with~(\ref{eq43568}) that
		\begin{align*}
			\|I_1-I_2\| &\leq \sum_{k=l}^n \frac{\gamma_lb_k}{b_l} \|\cH_n[l,k]-e^{(t_k-t_l)H_n}\|\\
			&\leq C' \gamma_{l}\sum_{k=l}^n  \frac{\gamma_lb_k}{b_l\gamma_k} e^{-(t_{k}-t_l)L} (t_k-t_l) \gamma_k\\
			&=  C' \gamma_{l}\sum_{k=l}^n \int_{t_{k-1}-t_{l-1}}^{t_k-t_{l-1}} F_l(s) e^{-(t_{k}-t_l)L} (t_k-t_l) \, ds.
		\end{align*}
		Each integral is taken over an interval $(t_{k-1}-t_{l-1}, t_k-t_{l-1}]$ and for the respective  $s$  we get
		$$
		t_k-t_l \leq t_{k-1}-t_{l-1}\leq s \text{ \ and \ }  t_k-t_l = t_k-t_{l-1}-\gamma_l \geq s-\gamma_l.
		$$
		Thus
		\begin{align*}
			\|I_1-I_2\| \leq C' e^{\gamma_1 L} \gamma_l \int_0^{t_n-t_{l-1}} F_l(s) e^{-sL}s \, ds ,
		\end{align*}
		By Lemma~\ref{le:dom_tech}  there exists an integrable majorant for the latter integrand. Hence $\|I_1-I_2\|$  is uniformly bounded and converges to zero as $l,n\to\infty$ with $l\leq n$.

		We analyse $\|I_2-I_3\|$. One has
		\begin{align*}
			I_2-I_3&= \sum_{k=l}^n \Bigl(\frac{\gamma_l b_k}{b_l \gamma_k}-1\Bigr) \gamma_k e^{(t_k-t_l) H_n}= \sum_{k=l}^n \int_{t_{k-1}-t_{l-1}}^{t_k-t_{l-1}} \bigl(F_l(s)-1\bigr)  e^{(t_k-t_l) H_n}\, ds
		\end{align*}
		and  using that $\|e^{(t_k-t_l)H_n}\|\leq e^{-L(t_k-t_{l})}$ we argue as before to get that
		$$
		\|I_2-I_3\|\leq e^{\gamma_1L} \int_0^{t_n-t_{l-1}} |F_l(s)-1|\, e^{-Ls} \, ds.
		$$
		Again there exists an integrable majorant. Hence $\|I_2-I_3\|$ is uniformly bounded and with dominated convergence and Lemma~\ref{le:dom_tech} we conclude that the latter integral converges to zero as $l,n\to\infty$ with $l\leq n$.

		We analyse $\|I_3+H_n^{-1}\|$.
		Using that $H_n^{-1}= - \int_0^\infty e^{s H_n}\,ds$ we write
		\begin{align*}
			I_3+H_n^{-1} =\sum_{k=l}^n \int_{t_{k-1}-t_{l-1}}^{t_k-t_{l-1}} (e^{(t_k-t_l)H_n}-e^{s H_n}) \, ds - \int_{t_n-t_{l-1}}^\infty e^{s H_n}\,ds.
		\end{align*}
		For $s\in  (  t_{k-1}-t_{l-1},t_k-t_{l-1} ]  $
		\begin{align*}
			\|e^{(t_k-t_l)H_n}-e^{s H_n}\|&=\|e^{(t_{k}-t_{l})H_n} (\1- e^{(s-(t_{k}-t_{l})H_n)})\|\\
			&\leq e^{-(t_k-t_l)L} (s-(t_{k}-t_{l})) C e^{(s-(t_{k}-t_{l}))C}\leq C'' e^{-Ls}  \gamma_l 
		\end{align*}
		with $C''=C e^{(C+L)\gamma_1}$.
		Hence, we get with $\|e^{s H_n}\| \leq e^{-Ls}$
		\begin{align*}
			\|I_3+H_n^{-1} \|& \leq C'' \gamma_l \int_0^{t_n-t_{l-1}} e^{-Ls} \, ds + \int_{t_n-t_{l-1}}^\infty e^{-Ls} \,ds \leq \sfrac{C''}L \gamma_l + \sfrac 1L e^{-L(t_n-t_{l-1})}.
		\end{align*}
		Letting $l,n\to \infty$ with $t_n-t_l\to \infty$ the previous term tends to zero.
		
		Altogether it thus follows that
		$$
		\limsup_{l,n\to\infty,\, t_n-t_l\to \infty} \underbrace{\bigl\|\bar \cH_{n}[l,n]+ H^{-1}\bigr\|}_{\leq \|I_1-I_2\|+\|I_2-I_3\|+\|I_3+H_n^{-1}\|+\|H_n^{-1}-H^{-1}\|}=0.
		$$

		(2) By Lemma~\ref{le:dom_tech}, there exists $N_{(\ref{lem:Huniformly})} \in \N$ and a measurable function $\bar F$ such that $F_l \le \bar F$ for all $l \ge N_{(\ref{lem:Huniformly})}$ with 
		$$
		\int_{0}^{\infty} \bar F(s) e^{-Ls} d s<\infty.
		$$
		By possibly increasing $N_{(\ref{lem:Huniformly})}$ we can guarantee that $\gamma_l<\frac{1}{C}$ for all $l \ge N_{(\ref{lem:Huniformly})}$. Note that estimate~(\ref{eq874574}) in step (1) prevails for arbitrary symmetric matrices $H$ with $\sigma(H)\subset [-C,-L]$. Hence for   $l,n\in\N$ with $N_{(\ref{lem:Huniformly})} \le l \le n$
		\begin{align*}
			\|\bar \cH[l,n]\| &= \Bigl\|\sum_{k=l}^{n}\frac{\gamma_l b_k}{b_l}\cH[l,k]\Bigr\| \le \sum_{k=l}^{n} \frac{\gamma_l b_k}{b_l } e^{-L(t_k-t_l)\gamma_k} \\
			&= \sum_{k=l}^{n} \int_{t_{k-1}-t_{l-1}}^{t_k-t_{l-1}} \frac{\gamma_l b_k}{b_l \gamma_k } e^{-L(t_k-t_l)\gamma_k}\, ds    \le e^{\gamma_1 L} \int_{0}^{t_n-t_{l-1}} F_l(s) e^{-sL} ds \\
			&\leq  e^{\gamma_1 L}\int_{0}^{\infty} \bar F(s) e^{-Ls} ds < \infty
		\end{align*}
		which proves uniform boundedness.
	\end{proof}

We are now in the position to prove the main result of this section.

	\begin{proof}[Proof of Theorem~\ref{thm:mainlin}]
		For $N\in\N$, $L,C\in(0,\infty)$ and  $n\geq N$ we consider the events
		$$
		\IA_{N..n,C,L}=\{\sigma(H_m)\subset [-C,-L] \text{ for $m=N,\dots, n$}\}  \text{ \ and \ } \ \IA_{N..\infty,C,L}=\bigcap\limits_{m \geq  N}\IA_{N..m,C,L}.
		$$
		We will use Theorem~\ref{thm:martingalecentrallimit} to verify the statement on the event $\IA_{N..\infty,C,L} \cap A$. 	By assumption, $H_n\to H$, almost surely, on $A$, so that in particular, almost surely, on $A$,
		$$
		\min \sigma(H_n)\to \min \sigma(H) \text{ \ and \ } \max \sigma(H_n)\to \max \sigma(H)<0.
		$$
		Hence, up to nullsets,
		$$
		A\subset \bigcup_{N,r,l\in\N} \IA_{N..\infty, r,\frac 1l}.
		$$
		It thus suffices to prove the statement on $A\cap \IA_{N..\infty,C,L}$ for fixed $N\in\N$ and $C,L>0$, see Lemma~\ref{lem:stableonsubsets}, and we briefly write for $n\geq N$
		$$
		\IA_n=\IA_{N..n,C,L}\text{ \ and \ } \IA_\infty=\IA_{N..\infty,C,L}.
		$$

		We denote by $N_{(\ref{lem:Huniformly})}$ and $C_{(\ref{lem:Huniformly})}$ the respective constants appearing in the second statement of Lemma~\ref{lem:Huniformly} and   restrict attention to $n\in\N$ with  $n_0(n)\geq N_{(\ref{lem:Huniformly})}\vee N$. 
		%
		We will apply Theorem~\ref{thm:martingalecentrallimit} with $(Z^{(n)}_m)_{m=1,\dots,n}$  given by
		$$
		Z_m^{(n)}= \1_{\IA_{n_0(n)}}\,\1_{\{ m>n_0(n)\}}\, \frac{b_m}{\bar b_n\sigma_n}  \bar \cH_n[m,n]\, \cD_m,
		$$
		%
		and with $A$ and $\Gamma$ in the Lemma replaced by $A\cap \IA_\infty$ and  
		$H^{-1} \Gamma(H^{-1})^\dagger$, respectively.  
		Once we have verified that Theorem~\ref{thm:martingalecentrallimit} is applicable we conclude that,  on $A\cap \IA_\infty$,
		$$
		\frac1{\sigma_n}\Xi_n= \sum_{m=1}^n Z_m^{(n)} \stackrel {\mathrm {stably}}\Longrightarrow H^{-1} \cN(0, \Gamma)
		$$
		which finishes the proof.
		

		
		It remains to verify the assumptions of Theorem~\ref{thm:martingalecentrallimit}. For $m=1,\dots,n_0(n)$ we have $Z^{(n)}_m=0$ and for $m=n_0(n)+1,\dots,n$, $ \1_{\IA_{n_0(n)}} \bar \cH_n[m,n]$ is $\cF_{n_0(n)}$-measurable and hence $\cF_{m-1}$-measurable, and uniformly bounded by  $C_{(\ref{lem:Huniformly})}$. 
		Consequently,  $(Z_m^{(n)})_{m=1,\dots,n}$ is a sequence of martingale differences satisfying for $\eps>0$
		\begin{align*}
			\sigma_n^{-2} \sum_{m=n_0(n)+1}^n \E\Bigl[&\1_{\IA_{n_0(n)}}\1_{\{|\frac {b_m}{\bar b_n}  \bar \cH_n[m,n] \cD_m| / \sigma_n\geq \eps\}} \Bigl|\frac {b_m}{\bar b_n}  \bar \cH_n[m,n] \cD_m\Bigr|^2\Big|\cF_{m-1}\Bigr] \\
			&\leq (C_{(\ref{lem:Huniformly})})^2 \sigma_n^2 \sum_{m=n_0(n)+1}^n \frac{b_m^2} {\bar b_n^{2}} \E\Bigl[\1_{\{ | \cD_m| \geq \frac {\eps\bar b_n \sigma_n }{C_{(\ref{lem:Huniformly})} b_m}\}}
			| \cD_m|^2\Big|\cF_{m-1}\Bigr]
		\end{align*}
		and the latter term tends to zero in probability on $A$, by assumption.
		
		It remains to control the asymptotics of
		$$
		V_n:= 	 \sum_{m=1}^{n} \cov(Z_m^{(n)}|\cF_{m-1}) = \sigma_n^{-2} \sum_{m=n_0(n)+1}^n \frac{b_m^2\delta_m^2}{\bar b_n^2} \1_{\IA_{n_0(n)}} \bar \cH_n[m,n] \cov(\delta_m^{-1}\cD_m| \cF_{m-1}) \bar\cH_n[m,n]^\dagger
		$$
		on $A\cap\IA_\infty$.
		By Lemma~\ref{le:count_set}, we can choose a sequence $(L(n))$ such that $n_0(n)\leq L(n)$, $$t_n-t_{L(n)}\to \infty\text{ \  and \ } n-L(n)=o(n).$$
		%
		%
		Now, by assumption,
		$$
		\sum_{m=L(n)+1}^n (b_m\delta_m)^2=o((\sigma_n \bar b_n)^2).
		$$
		As consequence of~Assumption (3.a) 
		$$
		\kappa:=\sup_{m\geq N}  \bigl\|\cov(\delta_m^{-1} \cD_m|\cF_{m-1})\bigr\|
		$$
		is almost surely  finite on $A$. We thus get that on $A\cap\IA_\infty$
		\begin{align}\begin{split}\label{eq858983}
				\Bigl\| 	&\sum_{m=L(n)+1}^{n} \bigl(\cov(Z_m^{(n)}|\cF_{m-1})-\sigma_n^{-2}  \frac{b_m^2\delta_m^2}{\bar b_n^2} H^{-1} \Gamma (H^{-1})^\dagger\bigr)  \Bigr\|\\
				&\leq \sigma_n^{-2} \sum_{m=L(n)+1}^n  \frac{b_m^2\delta_m^2}{\bar b_n^2}   \Bigl(\bigl\|  \bar \cH_n[m,n] \,\cov(\delta_m^{-1}\cD_m| \cF_{m-1}) \, \bar \cH_n[m,n]^\dagger\bigr\|+ \bigl\| H^{-1} \Gamma H^{-1}\bigr\|\Bigr) \\
				&\leq 2\kappa C_{(\ref{lem:Huniformly})}^2 (\sigma_n\bar b_n)^{-2} \sum_{m=L(n)+1}^n  (b_m\delta_m)^2   
				\to 0\text{, \ almost surely, on $A\cap \IA_\infty$.}
		\end{split}\end{align}
		
		By assumption 
		$$
		\rho_n:= \sup_{m=n_0(n)+1,\dots,n} \bigl\|\cov( \delta_n^{-1}  \cD_m|\cF_{m-1}) -\Gamma\bigr\| \to 0\text{, \ almost surely, on $A$,}
		$$
		and by Lemma~\ref{lem:Huniformly},
		$$
		\rho_n':= \sup_{m=n_0(n)+1,\dots,L(n)} \bigl\|\bar \cH_n[m,n] +H^{-1}\bigr\|\to 0\text{, \ almost surely, on $A$.}
		$$
		Consequently,
		one has for $m=n_0(n)+1,\dots, L(n)$, on $A\cap \IA_\infty$,
		\begin{align*}
			\| \bar \cH_n[m,n] &\cov(\delta_i^{-1} \cD_m|\cF_{m-1}) \bar \cH_n[m,n]^\dagger-   H^{-1}\Gamma (H^{-1})^\dagger \|\\
			&\leq \|\bar \cH_n[m,n]+H^{-1}\| \| \cov(\delta^{-1} \cD_m|\cF_{m-1})\| \| \bar \cH_n[m,n]^\dagger\| \\
			&+ \|H^{-1}\| \|\cov(\delta_m^{-1} \cD_m|\cF_{m-1}) - \Gamma\| \|\bar \cH_n[m,n]^\dagger\|+  \| H^{-1}\| \|\Gamma\| \|\bar \cH_n[m,n]^\dagger+ (H^{-1})^\dagger\|\\
			&\leq 2 \kappa  C_{(\ref{lem:Huniformly})} \rho_n' +(C_{(\ref{lem:Huniformly})})^2 \rho_n
		\end{align*}
		and thus, on $A\cap \IA_\infty$,
		\begin{align*}
			\Bigl\| 	&\sum_{i=n_0(n)+1} ^{L(n)} \bigl(\cov(Z_i^{(n)}|\cF_{i-1})-\sigma_n^{-2}  \frac{b_i^2\delta_i^2}{\bar b_n^2} H^{-1} \Gamma (H^{-1})^\dagger\bigr)  \Bigr\| \leq 2 \kappa  C_{(\ref{lem:Huniformly})} \rho_n' +(C_{(\ref{lem:Huniformly})})^2 \rho_n \to 0\text{, \ almost surely.}
		\end{align*}
		By definition, one has $\sigma_n^{-2}\sum _{i=n_0(n)+1} ^{n}  \frac{b_i^2\delta_i^2}{\bar b_n^2}=1$ so that together with~(\ref{eq858983}) we obtain that on $A\cap \IA_\infty$
		$$
		\bigl\| 	V_n-  H^{-1} \Gamma (H^{-1})^\dagger\bigr)  \bigr\|\to 0\text{, \ almost surely.}
		$$
		This finishes the proof.
\end{proof}

\begin{rem}
	Theorem~\ref{thm:mainlin} remains true  when replacing $(\sigma_n)$ by $(\sigma'_n)$  given by
	$$
	\sigma'_n=\frac 1{\bar b_n}\sqrt {\sum_{l=1}^n (b_l\delta_l)^2}
	$$
	and $(n_0(n))_{n \in \N}$ being  a sequence with 
	$$
	\sum_{i=1}^{n_0(n)}b_i^2\delta_i^2 = o\Bigl(\sum_{i=1}^{n}b_i^2\delta_i^2\Bigr).
	$$
	Indeed, in that case we have
	$$
	\frac{\sigma_n^2}{(\sigma'_n)^2}= 1- \frac{\sum_{i=1}^{n_0(n)}b_i^2\delta_i^2}{\sum_{i=1}^{n}b_i^2\delta_i^2} \to 1.
	$$
\end{rem}

\section{Technical preliminaries} \label{sec:prop}

In this section, we provide some technical estimates. 
First we deduce that the notion of a regular function entails certain Taylor type error estimates. 
Technically, we need to take care of the fact that segments connecting two points are not necessarily contained in the domain of the function.

	\begin{lemma}\label{rem:8435}
	Let $U\subset \R^d$ be an open and bounded  set, $g:U\to \R^d$ be a mapping and $\alpha_g\in(0,1]$.
	\begin{enumerate} \item If $g$ has regularity $\alpha_g$, then for every $\delta>0$ there exists a constant $C_g$ such that for all $x,y\in U_\delta=\{z\in U:d(z,U^c)>\delta\}$
		\begin{enumerate}
			\item $|g(x)|\vee \|Dg(x)\|\leq C_g$
			\item  $|g(y) - (g(x)+Dg(x) (y-x)) |\leq C_g |y-x|^{1+\alpha_g}$  and 
			\item $\|D g(y)-D g(x) \|\leq C_g |y-x|^{\alpha_g}.$\smallskip
		\end{enumerate}
		\item 
		If  $g:U\to \R^d$ has
		regularity $\alpha_g$ around a subset $M\subset \R^d$, then there exists a constant~$C_g$ such that for all $x\in M\cap U$ and $y\in U$ 
		\begin{enumerate} \item $|g(x)|\vee \|Dg(x)\|\leq C_g$ and
			\item  $|g(y) - (g(x)+Dg(x) (y-x)) |\leq C_g |y-x|^{1+\alpha_g}$
		\end{enumerate} and for all  $x,y\in M\cap U$
		\begin{enumerate}
			\item[(c)] $\|D g(y)-D g(x) \|\leq C_g |y-x|^{\alpha_g}.$
		\end{enumerate}
		\end{enumerate}
		\end{lemma}
		
		\begin{proof}
		First we prove (1).
		$g$ is continuous and thus bounded on the compact set $\overline{U_\delta}\subset U$ so that properties (a) and (c) follow from the  H\"older continuity of Dg and the boundedness of $U$. By Taylor's formula property (b) holds for every $x,y\in U$ with the constant $\sup \|Dg\|$ whenever  the segment connecting $x$ and $y$ lies in $U$. Now suppose that properties (a) and (c) are true for the constant $C$ and that $\sup_{x,y\in U}d(x,y)\leq  C$. We consider two points $x,y\in U_\delta$  whose segment is \emph{not} contained in $U$. Then we have that $d(x,y)\geq 2\delta$ so that
		\begin{align*}
			|g(y)-(g(x)+Dg(x)(y-x))|& \leq |g(y)|+|g(x)|+|Dg(x)(y-x))|\\&\leq 2 C + C^2
			\leq \frac {2 C + C^2}{(2\delta)^{1+\alpha_g}} |y-x|^{1+\alpha_g}.
		\end{align*}
		Consequently, properties (a), (b) and (c) are true on $U_\delta$ for a sufficiently large constant~$C_g$.
	
		The proof of (2) is straight-forward. Note that properties (b) and (c) are true for a sufficiently large constant and that (a) follows with the boundedness of $U$ and (b).
\end{proof}

Let now  $U$  denote an $(F,M)$-attractor with stability $L>0$ and bound $C$ and suppose that $\Phi:U_\Phi\to U$ is a nice representation for $M$ on $U$ of regularity $(\alpha_f,\alpha_\Phi,\alpha_\Psi)$ with $\alpha_f,\alpha_\Phi,\alpha_\Psi\in(0,1]$. 
We fix $\delta>0$  and choose $\rho\in(0,\delta/4]$ as in (i) of Lemma~\ref{lem:Geometry} and again denote by $U_\delta^\rho$ the set
$$
U_\delta^\rho=\bigcup_{\substack{y\in M:\\d(y,U^c)>\delta}} B_\rho (y).
$$
Recall that by  Lemma~\ref{lem:Geometry}, for every $x\in U_\delta^\rho$ there exists a unique closest element $x^*$ in $M$ and one has
$$
	x^*=\Phi(\Psi_\zeta(x),0) \in U\cap M.
$$

Now let  $(X_n)$ and $(\gamma_n)$ as introduced in  (\ref{eq:dynsys1}). 
We analyse the dynamical system based on the  nice representation introduced above. That means, 
for every $n \in \N$, we define on the event  $\{X_n \in U\}$  the coordinates
\begin{align} \label{def:coordinates}
\begin{pmatrix} \zeta_n \\ \theta_n\end{pmatrix} =\Psi(X_n)=\begin{pmatrix} \Psi_\zeta(X_n)\\ \Psi_\theta(X_n)\end{pmatrix},
\end{align}
where
$$
	\Psi_\zeta(x)= \begin{pmatrix} \Psi_1(x)\\  \vdots \\ \Psi_{d_\zeta}(x)\end{pmatrix} \quad , \quad \Psi_\theta(x)=\begin{pmatrix} \Psi_{d_{\zeta+1}}(x)\\  \vdots \\ \Psi_d(x)\end{pmatrix}.
$$

Crucial in our approach is the analysis of a linearised system. For a fixed element $\bar x = \Phi(\bar \zeta,0)\in M\cap U$ and every $n\in\N$ we define on the event that  $X_{n-1}$ and $X_n$ both are in $U$ the random variable $\Upsilon_n$ via 
\begin{align}\label{eq:Rep_Ups}
	\begin{pmatrix} \zeta_{n}\\ \theta_{n}\end{pmatrix}= 
	\begin{pmatrix} \zeta_{n-1}\\ \theta_{n-1}\end{pmatrix} +\gamma_n \Bigl (\begin{pmatrix} 0\\ H_{\bar x} \theta_{n-1}\end{pmatrix}  +\Upsilon_n\Bigr),
\end{align}
where $H_{\bar x}$ is the matrix with
$$ H_{\bar x} \theta =D\Psi_\theta(\bar x) Df(\bar x) (D\Psi (\bar x))^{-1} \begin{pmatrix} 0\\ \theta\end{pmatrix}.
$$
Informally, 
$$
\Upsilon_n= D\Psi(X_{n-1}) D_n+ \text{error term}
$$
and we control the error term in the following lemma.

	\begin{lemma}\label{le:lin_err}
		Suppose that $\Phi:U_\Phi\to U$ is a nice representation for $M$ on a bounded and open set $U$  
		 with regularity $(\alpha_f,\alpha_\Phi,\alpha_\Psi)\in(0,1]^3$. Let $\delta>0$ and $\rho\in(0,\delta/4]$ as in (i) of Lemma~\ref{lem:Geometry}.
		There exists a  constant $C^{(\ref{le:lin_err})}$ such that the following is true. If  for $x \in U_\delta^\rho$, $\gamma\in(0,\gamma_0]$, $u\in \R^d$ one has
		$$
		x':=x+\gamma(f(x)+u) \in  U_\delta,
		$$
		then for every $\bar x\in M\cap U$, $\theta= \Psi_\theta(x)$ and $\Upsilon\in \R^d$ given by
		\begin{align} \label{equ:Upsilon}
			\Psi(x')-\Psi(x)= \gamma  \Bigl (\begin{pmatrix} 0\\ H_{\bar x} \theta\end{pmatrix}  +\Upsilon\Bigr)\text{, \ where \ }
			H_{\bar x} \theta =D\Psi_\theta(\bar x) Df(\bar x) (D\Psi (\bar x))^{-1} \begin{pmatrix} 0\\ \theta\end{pmatrix},
		\end{align}
		one has 
		$$
		|\Upsilon - D\Psi(x) u| \leq  C^{(\ref{le:lin_err})} \bigl(\gamma^{\alpha_\Psi }|u|^{1+\alpha_\Psi } + d(x,M) d(x,\bar x)^{\alpha}\bigr),
		$$
		where $\alpha=\alpha_\Psi\wedge\alpha_f\wedge \alpha_\Phi$.
	\end{lemma}

	\begin{proof}
Note that by assumption $x$, $x'$ and $x^*$	 are all in $U_{\delta/2}$ and we will use the Taylor-type estimates of Lemma~\ref{rem:8435} without further mentioning.
	 For $\bar x \in M\cap U$ we set $\bar H_{\bar x}= D\Psi(\bar x) Df( \bar x) (D\Psi ( \bar x))^{-1}$. Then
		$$
		\begin{pmatrix} 0\\ H_{\bar x}  \theta\end{pmatrix}= \bar H_{\bar x} \begin{pmatrix} 0\\  \theta\end{pmatrix}
		$$
		since a vector $\begin{pmatrix}0\\ \theta\end{pmatrix}$ is mapped by $(D\Psi ( \bar x))^{-1}=D\Phi(\Psi(\bar x))$ to a vector in $N_z M$ which is mapped itself by $Df(z)$  to a vector in $N_zM$ (see Remark~\ref{rem:8456}) and then by $D\Psi (z)$ to a vector in $\{0\}^{d_\zeta}\times \R^{d_\theta}$. As consequence of~(\ref{equ:Upsilon}) we get that
		$$
		\Upsilon =\frac 1{\gamma}(\Psi(x+\gamma(f(x)+u))-\Psi(x))- \bar H_{\bar x} \begin{pmatrix} 0\\  \theta\end{pmatrix}.
		$$
		Using the $\alpha_\Psi$-regularity of $\Psi$ we get that 
		\begin{align*}
			\frac 1{\gamma} &(\Psi(x+\gamma(f(x)+u))-\Psi(x))= D \Psi(x) (f(x)+u) + \cO(\gamma^{ \alpha_\Psi} (|f(x)|^{1+\alpha_\Psi} +|u|^{1+\alpha_\Psi })).\end{align*} 
		Here and elsewhere in the proof all $\cO$-terms are uniform over all allowed choices of $x$, $x'$, $\bar x$ and~$\gamma$. 
		%
		By Lemma~\ref{lem:Geometry},  $x$ has a unique closest $M$-element $x^*\in U\cap M$  and using the $\alpha_f$ regularity of $f$ and the boundedness of $U$ we get that  
		$$
		|f(x)|=|f(x)-f(x^*)|=\cO(|x-x^*|)= \cO(d(x,M)).
		$$
		Hence,
		\begin{align} \label{equ:OTerms1}
			\frac 1{\gamma} (\Psi(x+\gamma(f(x)+u))-\Psi(x))= D \Psi(x) (f(x)+u) + \cO(\gamma^{ \alpha_\Psi} (d(x,M)^{1+\alpha_\Psi} +|u|^{1+\alpha_\Psi })).
		\end{align} 
		and with the $\alpha_\Psi$-regularity of $\Psi$  we get
		\begin{align} \label{equ:OTerms2}
			D \Psi(x) f(x)&= D \Psi(x^*) f(x)  +\cO (|x-x^*|^{\alpha_\Psi} |f(x)|)=  D \Psi(x^*) f(x)  +\cO (d(x,M)^{1+\alpha_\Psi } ).
		\end{align}
		Furthermore, Lemma~\ref{lem:Geometry} yields that $|\theta| = d(x,M)$, so that with $f(x^*)=0$
		\begin{align} \begin{split} \label{equ:OTerms3}
				D \Psi(x^*) f(x) &=D \Psi( x^*) Df(x^*) (x- x^*)+\cO(d(x,M)^{1+\alpha_f })\\
				&= \underbrace{D \Psi( x^*) Df( x^*) (D \Psi( x^*))^{-1}}_{=\bar H_{ x^*}  } \begin{pmatrix} 0 \\ \theta\end{pmatrix}+\cO(d(x,M)^{1+\alpha_f }+d(x,M)^{1+\alpha_\Phi }).
			\end{split}
		\end{align}
		Insertion of~(\ref{equ:OTerms1}), (\ref{equ:OTerms2}) and (\ref{equ:OTerms3}) into the above representation of $\Upsilon$ gives together with the uniform boundedness of $\gamma$ and $d(x,M)$
		$$
		\Upsilon= D\Psi(x) u +(\bar H_{x^*}-\bar H_{\bar x}) \begin{pmatrix}0\\ \theta\end{pmatrix} + \cO(d(x,M)^{1+\alpha}+ \gamma^{\alpha_\Psi} |u|^{1+\alpha_\Psi}).
		$$
		On the relevant domains $D\Psi$, $Df$ and $D\Phi$ are H\"older  continuous with parameter~$\alpha$ and uniformly bounded so that $\|\bar H_{x^*}-\bar H_{\bar x}\|=\cO(|x^*-\bar x|^\alpha)$. Since $|\theta|=d(x,M)$ we finally get that
		\begin{align*}
			\Upsilon&= D\Psi(x) u +\cO(d(x,M) |\bar x-x^*|^\alpha +d(x,M)^{1+\alpha}+ \gamma^{\alpha_\Psi} |u|^{1+\alpha_\Psi})\\
			&=D\Psi(x) u +\cO(d(x,M) |x- \bar x|^\alpha + \gamma^{\alpha_\Psi} |u|^{1+\alpha_\Psi}).
		\end{align*}
	\end{proof}

	\begin{prop} \label{prop:zeta_est} 
			Let $U$  be a $(F,M)$-attractor with stability $L>0$ and bound~$C$ and suppose that $\Phi:U_\Phi\to U$ is a nice representation for $M$ on  $U$  
		 with regularity $(\alpha_f,\alpha_\Phi,\alpha_\Psi)\in(0,1]^3$. Set $\alpha=\alpha_\Psi\wedge\alpha_f\wedge \alpha_\Phi$. Let $(X_n)$ be as in (\ref{eq:dynsys1}) satisfying the following assumptions:
		 \begin{itemize}\item   $(\1_{U}(X_{n-1})  D_n)_{n\in\N}$ is  a sequence of square-integrable martingale differences,
		 \item 	$(\gamma_n)$ is a sequence of strictly positive reals with 	 $\gamma_n\to0$ and $\sum \gamma_n=\infty$,
		 \item   $(\sigma^{\mathrm{RM}}_n)_{n\in\N}$ is a sequence of strictly positive reals with
		\begin{align}\label{eq94657}
		L'':=\limsup_{n\to\infty} \frac 1{\gamma_n}\frac{\sigma^{\mathrm{RM}}_{n-1}-\sigma^{\mathrm{RM}}_n}{\sigma^{\mathrm{RM}}_n} <L \text{ \ and \ }\limsup_{n\to\infty} \bigl(\frac{\sigma_n^\mathrm{RM}}{\sqrt{\gamma_n}}\bigr)^{-1}\, \E\bigl[\1_U(X_{n-1}) |D_n|^2\bigr]^{1/2}<\infty.
		\end{align}
		\end{itemize}	
		Let $\delta>0$ and $\rho\in(0,\delta/4]$ be as in (i) of Lemma~\ref{lem:Geometry} and suppose that  inequality~(\ref{eq190206-1}) of Prop.~\ref{prop:dist_est} is true on $U_\delta^\rho$  for a $L' \in (L'',L)$ that is $d(x+\gamma f(x),M)\leq(1-\gamma L')d(x,M)$ for all $x\in U_\delta^\rho$ and $\gamma\in[0,C^{-1}]$.
	Then for every $N\in\N$, as $n\to\infty$,
		$$
		\sup_{m=n_0(n)+1,\dots,n } |\zeta_m-\zeta_{n_0(n)}| = \cO_P\bigl( \eps_n^{(\ref{prop:zeta_est})}\bigr) \text{ , \ on } \IU_{N..\infty}^{\delta,\rho},
		$$
		where $\zeta_m$ ($m\in\N_0$) is well-defined via~(\ref{def:coordinates}) on $\{X_m\in U\}$,
		\begin{align}\label{eq95467}
			\eps_n ^{(\ref{prop:zeta_est})} =\sum_{k=n_0(n)+1}^n \bigl((\sqrt{\gamma_k}\sigma^{\mathrm{RM}}_k)^{1+\alpha_\Psi}+\gamma_k(\sigma^{\mathrm{RM}}_{k-1})^{1+\alpha}\bigr)+\sqrt{\sum_{k=n_0(n)+1}^n \gamma_k  {(\sigma^{\mathrm{RM}}_k)^2}}
		\end{align}
		and 
		$$
		\IU_{N..n}^{\delta,\rho} = \{\forall l=N,\dots, n: X_l \in U_\delta^\rho \} \text{ \ and \ } \IU_{N..\infty}^{\delta,\rho}= \bigcap\limits_{n \ge N}\IU_{N..n}^{\delta,\rho}.
		$$
	\end{prop}
	
	\begin{proof}By Theorem \ref{thm:L2_bound}, 
	there exists a constant $C_{(\ref{thm:L2_bound})}$ such that for all $n\in\N$
		\begin{align}\label{eq94563}
			\limsup_{n\to\infty} (\sigma^{\mathrm{RM}}_n)^{-1} \E\bigl[\1_{\IU_{N..n-1}^{\delta,\rho}} d(X_n,M)^2\bigr]^{1/2}\leq C_{(\ref{thm:L2_bound})}.
		\end{align}
		
		We fix $N\in\N$ and briefly write  $\IU_k=\IU_{N..k}^{\delta,\rho}$ for $k\geq N$.
		By choice of $\rho$, Lemma~\ref{le:lin_err} is applicable on $U_\delta ^\rho$ and we conclude that  for all  $m$ for which   $X_{m-1}$ and $X_m$ lie in $U_\delta^\rho$ we have
		$$
		\zeta_m=\zeta_{m-1}+ \gamma_m D\Psi_\zeta(X_{m-1}) D_m + \cO(\gamma_m^{1+\alpha_\Psi }|D_m|^{1+\alpha_\Psi } +\gamma_m d(X_{m-1},M)^{1+\alpha }).
		$$
		Here we used the lemma with  $x = X_{m-1}$, $x'=X_m$, $\bar x = X_{m-1}^*$ and $\gamma=\gamma_m$.
		Note that  the $\cO$-term is uniformly bounded over all realisations and allowed choices of $m$. 
		
		We  consider $n\in\N$ with $n_0(n)\geq N$. 
		On $\IU_n$, one has for $m=n_0(n)+1,\dots,n$,
		\begin{align*}
			\zeta_m -\zeta_{n_0(n)} &= \underbrace{\sum_{k=n_0(n)+1}^m \gamma_k D\Psi_\zeta(X_{k-1})  D_k}_{=:A_m^{(1)}} 
			+\underbrace{\cO\Bigl( \sum_{k=n_0(n)+1}^n \gamma_k^{1+\alpha_\Psi }|D_k|^{1+\alpha_\Psi } +\gamma_k d(X_{k-1},M)^{1+\alpha }\Bigr)}_{=:A^{(2)}_m}.
		\end{align*}
		For ease of notation we omit the $n$-dependence in the notation of the $A$-terms. We control 
		$$
		S^{(i)}_n := \sup_{m=n_0(n)+1,\dots, n} \,\bigl|     A^{(i)}_m\bigr|
		$$
		for the two choices of $i$ separately.
		
		By the boundedness of $D\Psi_\zeta$ the sequence $(\1_{\IU_{k-1}} \gamma_k D\Psi_\zeta(X_{k-1}) D_k)_{k=n_0(n)+1,\dots,n}$ defines a sequence of square integrable martingale differences. Hence we get with  Doob's martingale inequality, the uniform boundedness of $D\Psi_\zeta$ and~(\ref{eq94657}) that
		$$
		\E[|\1_{\IU_n} S_n^{(1)}|^2] \leq 4 C_\Psi \,\E\Bigl[ \sum_{k=n_0(n)+1}^n  \1_{\IU_{k-1}}\gamma_k^2  |D_k|^2\Bigr] = \cO\Bigl( \sum_{k=n_0(n)+1}^n \gamma_k  {(\sigma^{\mathrm{RM}}_k)^2}\Bigr).
		$$
		Hence,  
		$$
		S_n^{(1)}=\cO_P\Bigl( \sqrt{\sum_{k=n_0(n)+1}^n \gamma_k  {(\sigma^{\mathrm{RM}}_k)^2}}\Bigr) \text{, \ on \ }\IU_\infty,
		$$
		see Remark~\ref{rem:OPMarkov}.
		It remains to bound the second term. 
		
		Note that by assumption 
		\begin{align*}
			\E\Bigl[ \sum_{k=n_0(n)+1}^n \1_{\IU_{k-1}}\gamma_k^{1+\alpha_\Psi }|D_k|^{1+\alpha_\Psi }\Bigr] 
			&\leq  \sum_{k=n_0(n)+1}^n \gamma_k^{1+\alpha_\Psi } \E[\1_{\IU_{k-1}}|D_k|^2]^{(1+\alpha_\Psi)/2} \\
			&=\cO\Bigl(\sum_{k=n_0(n)+1}^n \bigl(\sqrt{\gamma_k}  {\sigma^{\mathrm{RM}}_k}\bigr)^{1+\alpha_\Psi }\Bigr)
		\end{align*}
		and with (\ref{eq94563})
		\begin{align*}
			\E\Bigl[ \sum_{k=n_0(n)+1}^n \gamma_k \1_{\IU_{k-1}} d(X_{k-1},M)^{1+\alpha }\Bigr] &\leq  \sum_{k=n_0(n)+1}^n \gamma_k \E\bigl[ \1_{\IU_{k-1}} d(X_{k-1},M)^2\bigr]^{(1+\alpha )/2} \\
			&=\cO\Bigl( \sum_{k=n_0(n)+1}^n \gamma_k (\sigma_{k-1}^\mathrm{RM})^{1+\alpha}\Bigr)
		\end{align*}
		so that (see again Remark~\ref{rem:OPMarkov})
		$$
		S_n^{(2)}= \cO_P\Bigl(\sum_{k=n_0(n)+1}^n \bigl(\bigl(\sqrt{\gamma_k}  {\sigma^{\mathrm{RM}}_k}\bigr)^{1+\alpha_\Psi }+ \gamma_k  (\sigma^{\mathrm{RM}}_{k-1})^{1+\alpha}\bigr)\Bigr).
		$$
		Together with the respective bound for $S_N^{(1)}$ above this finishes the proof of the proposition.
	\end{proof}

	\begin{prop}\label{prop:tec_2}
		We assume the same assumptions as in  Proposition~\ref{prop:zeta_est}. 
		Then for every $N\in\N$
		\begin{align}\label{eq784350}
			\frac1{\bar b_n} \sum_{k=n_0(n)+1}^n b_k (\gamma_k^{\alpha_\Psi} |D_k|^{1+\alpha_\Psi} +d(X_{k-1},M)d (X_{k-1}, X_{n_0(n)}^*)^\alpha) =\cO_P\bigl(\eps_n^{(\ref{prop:tec_2})}\bigr)\text{, \ on \ $\IU_{N..\infty}^{\delta,\rho}$,}
		\end{align}
		where
		$$
		\eps_n^{(\ref{prop:tec_2})} = \frac 1{\bar b_n} \sum_{k=n_0(n)+1}^n b_k \bigl(\gamma_k^{-\frac {1-\alpha_\Psi}2} (\sigma^{\mathrm{RM}}_k)^{1+\alpha_\Psi}+(\sigma^{\mathrm{RM}}_{ k-1})^{1+\alpha}+\sigma^{\mathrm{RM}}_{ k-1}(\eps_n^{(\ref{prop:zeta_est})})^\alpha\bigr).
		$$
	\end{prop}
	\begin{proof}
		Fix $N\in\N$, consider $n\in\N$ with $n_0(n)\geq N$ and briefly write $\IU_k=\IU_{N..k}^{\delta,\rho}$ for $k\geq N$.
		First note that with Lemma~\ref{lem:Geometry} and the convexity of $U_\Phi$ for $k> n_0(n)$, on $\IU_\infty$,
		\begin{align*}
			|X_{k-1}- X_{n_0(n)}^*|&\le |X_{k-1}-X_{k-1}^*|+|X_{k-1}^*-X_{n_0(n)}^*|\\&\le d(X_{k-1},M) + C_\Phi |\zeta_{k-1}, \zeta_{n_0(n)})|.
		\end{align*}
		Using this inequality the left hand side of~(\ref{eq784350}) is transformed into the sum of three terms that we will analyse independently below.

		1) First, we provide an asymptotic bound for
		$$\frac1{\bar b_n} \sum_{k=n_0(n)+1}^n b_k d(X_{k-1},M)d(\zeta_{k-1}, \zeta_{n_0(n)})^\alpha,
		$$
		on $\IU_\infty$. 
		By choice of $\rho$, we have validity of~(\ref{equ:L2conv}) and we get that
		\begin{align*}
			\E\Bigl[ \frac1{\bar b_n} \sum_{k=n_0(n)+1}^n b_k \1_{\IU_{k-1}}d(X_{k-1},M)\Bigr] &\leq \frac1{\bar b_n} \sum_{k=n_0(n)+1}^n b_k \E[\1_{\IU_{k-1}} d(X_{k-1},M)^2]^{1/2} \\
			&= \cO\Bigl(\frac1{\bar b_n} \sum_{k=n_0(n)+1}^n b_k \sigma_{k-1}^\mathrm{RM}\Bigr).
		\end{align*}
		Hence,
		$$\frac1{\bar b_n} \sum_{k=n_0(n)+1}^n b_k \1_{\IU_{k-1}}d(X_{k-1},M)=\cO_P\Bigl(\frac1{\bar b_n} \sum_{k=n_0(n)+1}^n b_k \sigma_{k-1}^\mathrm{RM}\Bigr)\text{, \ on $\IU_\infty$}.$$
		With   Proposition~\ref{prop:zeta_est} we conclude that, on $\IU_\infty$,
		\begin{align*}
			\frac1{\bar b_n}& \sum_{k=n_0(n)+1}^n b_k d(X_{k-1},M)d(\zeta_{k-1}, \zeta_{n_0(n)})^\alpha
			\\
			&\leq \Bigl (\frac1{\bar b_n} \sum_{k=n_0(n)+1}^n b_k d(X_{k-1},M)\Bigr) \sup_{m=n_0(n)+1,\dots,n} |\zeta_{m}- \zeta_{n_0(n)}|^\alpha \\
			&= \cO_P \Bigl( \frac 1{\bar b_n} \sum_{k=n_0(n)+1}^n b_k  \sigma^{\mathrm{RM}}_{ k-1}(\eps_n^{(\ref{prop:zeta_est})})^\alpha\Bigr)
		\end{align*}

		2) \emph{Analysis of the second term.} Second, we analyse
		$$
		\frac 1{\bar b_n}\sum_{k=n_0(n)+1}^n b_k d(X_{k-1},M)^{1+\alpha}.
		$$
		With (\ref{equ:L2conv}) we get that 
		\begin{align*}
			\E\Bigl[	\frac1{\bar b_n} \sum_{k=n_0(n)+1}^n b_k  \1_{\IU_{k-1}} d(X_{k-1},M)^{1+\alpha}\Bigr]& \le
			\frac1{\bar b_n} \sum_{k=n_0(n)+1}^n b_k  \E[ \1_{\IU_{k-1}} d(X_{k-1},M)^2]^{(1+\alpha)/2}\\
			&= \cO\Bigl(\frac1{\bar b_n} \sum_{k=n_0(n)+1}^n b_k  (\sigma_{k-1}^\mathrm{RM})^{1+\alpha}\Bigr)  
		\end{align*}
		so that
		$$
		\frac1{\bar b_n} \sum_{k=n_0(n)+1}^n b_k   d(X_{k-1},M)^{1+\alpha} =\cO_P\Bigl(\frac1{\bar b_n} \sum_{k=n_0(n)+1}^n b_k  (\sigma_{k-1}^\mathrm{RM})^{1+\alpha}\Bigr)\text{, \ on $\IU_\infty$.}
		$$
		
		3) \emph{Analysis of the third term.} 
		Similarly to before, we conclude that
		\begin{align*}
			\E\Bigl[\frac1{\bar b_n} \sum_{k=n_0(n)+1}^n b_k \gamma_k^{\alpha_\Psi} \1_{\IU_{k-1}} |D_k|^{1+\alpha_\Psi} \Bigr]
			&\leq \frac1{\bar b_n} \sum_{k=n_0(n)+1}^n b_k \gamma_k^{\alpha_\Psi} \E[\1_{\{X_{k-1}\in U\}} |D_k|^2]^{(1+\alpha_\Psi)/2} \\&=\cO\Bigl( \frac 1{\bar b_n} \sum_{k=n_0(n)+1}^n b_k \gamma_k^{-\frac {1-\alpha_\Psi}2} (\sigma^{\mathrm{RM}}_k)^{1+\alpha_\Psi}\Bigr)
		\end{align*}
		with the obvious $\cO_P$-bound on $\IU_\infty$.
		The statement is obtained by combining the three estimates.
\end{proof}

\section{The proofs of the main results} \label{sec:proof}

\subsection{Proof of Theorem~\ref{thm:MainCLT}}

\begin{proof}
1) \emph{Feasible triples.} Let $(U,\delta,\rho)$ be a feasible triple in the sense of Proposition~\ref{le:feasible}. We denote by $\Uconv=\Uconv_{\delta,\rho}$ the event that $(X_n)$ converges to some value in $M\cap U_\delta^\rho$. 
As explained in Remark~\ref{rem:feas84} the statement of Theorem~\ref{thm:MainCLT} follows once we showed stable convergence on $\Uconv$. 
%

Recall  that, by Lemma~\ref{lem:Geometry}, for all $x\in U_\delta^\rho$ there is a unique closest $M$-element  
$x^*=\Phi(\Psi_\zeta(x),0) \in M \cap U_\delta^\rho$.
For $m\in\N$ we define on the event $\{X_m\in U_\delta^\rho\}$  a random symmetric $d_\theta\times d_\theta$-matrix $H_m$ via
$$H_m\theta =D\Psi_\theta( X_{m}^* ) Df ( X_{m}^*) (D\Psi( X_{m}^*))^{-1}\begin{pmatrix} 0\\ \theta\end{pmatrix},$$
with symmetry following from Remark~\ref{rem:8456}.
For technical reasons, we set $H_m=0$ on $\{X_m\in U_\delta^\rho\}^c$.
Let $N\in\N$ and consider for $m\ge N$ the events 
$$
\IU_{N.. m}:=\{\forall l=N,\dots, m: X_l \in U_\delta^\rho\},
$$
$$
\IU_{N..\infty}:= \bigcap_{m'\geq N} \IU_{N..m'} \text{ \ and \ } \Uconv_{N..\infty}:=\Uconv\cap \IU_{N..\infty}.
$$
Note that 
$$
\Uconv =\bigcup_{N \in \N} \Uconv_{N..\infty}
$$
so that as consequence of Lemma~\ref{lem:stableonsubsets} it suffices to prove stable convergence on $\Uconv_{N..\infty}$ for arbitrarily fixed $N\in\N$.
Note that, on $\Uconv$, $(H_l)$ converges to the symmetric random matrix $H_\infty$ with
$$
H_\infty \theta =D\Psi_\theta( X_{\infty} ) Df ( X_{\infty}) (D\Psi( X_{\infty}))^{-1}\begin{pmatrix} 0\\ \theta\end{pmatrix}.
$$
 Set $A=D \Psi_\theta(X_\infty) \bigl (Df(X_\infty)\big|_{N_{X_\infty} M}\bigr)^{-1} \Pi_{N_{X_\infty}M}$
and note that by monotonicity it suffices to consider large $N$.
We briefly write
$$
\IU_{m}=\IU_{N.. m}, \ \IU_{\infty}=\IU_{N..\infty}\text{ \ and \ }  \Uconv_{\infty}= \Uconv_{N..\infty}.
$$

In the following we restrict attention to  $n\in\N$ with $n_0(n)\geq N$ and consider $m\geq n_0(n)$. Note that
$$
\bar \theta_{m}=\frac1{\bar b_m} \sum_{k={n_0(n)}}^m b_k \theta_k  \text{ \ and \ } \bar \zeta_{m}=\frac1{\bar b_m} \sum_{k={n_0(n)}}^m b_k \zeta_k 
$$
are well-defined on $\IU_m$.  
Moreover, for $m> n_0(n)$ we set on~$\IU_m$
$$
\Upsilon^{(n)}_m =\frac 1{\gamma_m}(\Psi(X_{m})-\Psi(X_{m-1}))-\begin{pmatrix} 0\\ H_{n_0(n)} \theta_{m-1}\end{pmatrix}
$$
and on~$\IU_m^c$,  $\Upsilon^{(n)}_m=0$. Now, on $\IU_m$,
$$
\theta_m=\theta _{m-1} + \gamma_m (H_{n_0(n)} \theta_{m-1}+\pi_\theta(\Upsilon^{(n)}_m))
$$
so that  by the variation of constant formula 
$$
\theta_m=\cH_{n_0(n)}[n_0(n),m] \theta_{n_0(n)} + \sum_ {\ell=n_0(n)+1}^m  \gamma_l  \cH_{n_0(n)}[\ell,m] \pi_\theta(\Upsilon^{(n)}_{ l}),
$$
with $\cH_{n_0(n)}[i,j]$ and $\bar \cH_{n_0(n)}[i,j]$ ($i,j \in \N$ with $i\le j$) being defined as in (\ref{eq:def_H}). \\
Consequently, on $\IU_n$,
\begin{align}\label{eq98458}
\bar \theta_n=\frac {b_{n_0(n)}}{\bar b_n\gamma_{n_0(n)}} \,\bar \cH_{n_0(n)}[n_0(n),n]  \theta_{n_0(n)} +\frac 1{\bar b_n} \sum_{m=n_0(n)+1}^nb_m \, \bar \cH_{n_0(n)} [m,n] \,\pi_\theta(\Upsilon_m^{(n)})
\end{align}
with the right hand side being a random variable that is defined on the whole space $\Omega$ and we take the previous formula as definition of the random variable $\bar\theta_n$ outside of $\IU_n$. For ease of notation we briefly write $\bar \cH[m,n]= \bar \cH_{n_0(n)}[m,n]$ for $m\leq n$.

3) \emph{Approximation by the linear system of Section~\ref{sec:linSys}.} We set 
$$\Xi_n:=\frac 1{\bar b_n} \sum_{m=n_0(n)+1} ^n b_m\, \bar \cH[m,n]\, \cD_m
$$
with $\cD_m=\1_{\IU_{m-1}} \, D\Psi_\theta(X_{m-1})D_m$. By Lemma~\ref{le:lin_err}, there exists a constant $C_{(\ref {le:lin_err})}$ such that, on $\IU_n$, for all $n_0(n)\le m\le n$
$$|\pi_\theta \Upsilon_m^{(n)}-\cD_m|\leq C_{(\ref {le:lin_err})} \bigl(\gamma_m^{\alpha_\Psi} |D_m|^{1+\alpha_\Psi}+d(X_{m-1},M) d(X_{m-1}, X_{n_0(n)}^*)^\alpha\bigr).
$$
Assuming that $N$ is sufficiently large, Lemma~\ref{lem:Huniformly} yields existence of a constant $C_{(\ref{lem:Huniformly})}$ such that,  on $\IU_n$, 
\begin{align*}
&  \Bigl|\Xi_n- \frac 1{\bar b_n} \sum_{m=n_0(n)+1}^nb_m \, \bar \cH[m,n] \,\pi_\theta(\Upsilon_m^{(n)})\Bigr| \\
&\le   C_{(\ref{lem:Huniformly})} C_{(\ref {le:lin_err})}   \frac1{\bar b_n} \sum_{m=n_0(n)+1}^n b_m \bigl(\gamma_m^{\alpha_\Psi} |D_m|^{1+\alpha_\Psi} +d(X_{m-1},M)d (X_{m-1}, X_{n_0(n)}^*)^\alpha\bigr).
\end{align*}
By Proposition~\ref{prop:tec_2}, the latter term is of order $\cO_P\bigl(\eps_n^{(\ref{prop:tec_2})}\bigr)$ on $\IU_\infty$. Thus assumption~(\ref{assu:eps1}) guarantees that the previous error term is of order $o_P(\sigma_n)$ on $\IU_\infty$.

4) \emph{Analysis of $\Xi_n$.}  Recall that on $\Uconv_\infty$, one has $
\lim_{n\to\infty} H_{n_0(n)} \to H_\infty$ with $H_\infty$ satisfying
$$ H_\infty\theta=D\Psi_\theta(X_\infty) Df(X_\infty)(D\Psi(X_\infty))^{-1}\begin{pmatrix} 0\\ \theta\end{pmatrix}
$$
By assumption $Df(X_\infty)$ as a linear mapping from $N_{X_\infty}M$ to $N_{X_\infty}M$ is invertible and we get with elementary linear algebra that for $\theta\in \R^{d_\theta}$
$$
H^{-1}_\infty \theta= D\Psi_\theta(X_\infty) \bigl(Df(X_\infty)|_{N_{X_\infty} M}\bigr)^{-1} (D\Psi(X_\infty))^{-1} \begin{pmatrix}0\\ \theta\end{pmatrix}
$$
Note that $(\cD_m)_{m\ge N+1}$ is  a sequence of martingale differences and one has, on $\Uconv_{\infty}$, for $m>N$,
\begin{align*}
\cov((\delta_m^{\mathrm{diff}})^{-1}\cD_m|\cF_{m-1})&= D\Psi_\theta(X_{m-1}) (\delta_m^{\mathrm{diff}})^{-2}\cov( D_m|\cF_{m-1})  D\Psi_\theta(X_{m-1})^\dagger  \\
&  \to D\Psi_\theta(X_\infty)\, \Gamma\, D\Psi_\theta(X_\infty)^\dagger\text{, \ almost surely}.
\end{align*}
Moreover,  assumption~(\ref{assu:D_n})  implies that for every $\eps >0$, on $\Uconv_{\infty}$,
\begin{align*}
	&\sigma_n^{-2} \sum_{m=n_0(n)+1}^n \frac {b_m^2}{\bar b_n^2} \E \bigl[\1_{\{|\cD_m|>\frac{\eps\bar b_n \sigma_n}{b_m}\}} |\cD_m|^2\big| \cF_{m-1}\bigr]\\
	& \le  (C_\Psi)^2(\sigma_n)^{-2} \sum_{m=n_0(n)+1}^n \frac {b_m^2}{\bar b_n^2} \E\bigl[\1_{\{|D_m|>\frac{\eps\bar b_n \sigma_n}{C_\Psi b_m}\}}|D_m|^2\big|\cF_{m-1}\bigr] \to 0, \text{ in probability.}
\end{align*}
Thus Theorem~\ref{thm:mainlin} implies that, on $\IU_{\infty}^\mathrm{conv}$,
$$
\frac 1{\sigma_n} \Xi_n \stab A \ \cN(0, \Gamma ).
$$
\ \\
Together with step 2 (see Lemma~\ref{lem:stablyOP}) we thus get that
$$
  \frac{1}{\sigma_n}\frac 1{\bar b_n} \sum_{k=n_0(n)+1}^nb_k \, \bar \cH_n[k,n] \,\pi_\theta(\Upsilon_k^{(n)}) \stab A \ \cN(0, \Gamma)  \text{, \ on \ $\Uconv_{\infty}$.}
$$

5) \emph{Analysis of the contribution  of $\theta_{n_0(n)}$.}
By choice of $\IU_n$ the asymptotic estimate (\ref{equ:L2conv}) holds. This entails together with property (ii) of Lemma~\ref{lem:Geometry} that, on $\IU_\infty$,
$$
|\theta_{n_0(n)}|=d(X_{n_0(n)},M)=\cO_P(\sigma^\mathrm{RM}_{n_0(n)}).
$$
Moreover, by Lemma~\ref{lem:Huniformly}, $\bar \cH[n_0(n),n]$ is uniformly bounded on $\IU_\infty$, so that, on $\IU_\infty$,
$$
\frac {b_{n_0(n)}}{\bar b_n\gamma_{n_0(n)}}  \bar \cH_{n_0(n)}[n_0(n),n]  \theta_{n_0(n)} =\cO_P\Bigl( \frac {b_{n_0(n)}}{\bar b_n\gamma_{n_0(n)}} \sigma^\mathrm{RM}_{n_0(n)}\Bigr)
$$
which is of order $o_P(\sigma_n)$ by assumption~(\ref{assu_step4}). With step 3 we thus obtain that, on $\Uconv_\infty$,
$$
\bar\theta_n \stab A \ \cN(0, \Gamma ).
$$

6) \emph{Comparison of $ \bar X_n$ and $\Phi( \bar \theta_n)$.} On $\Uconv_n$,
\begin{align*}
\bar X_n&= \frac1{\bar b_n} \sum_{m=n_0(n)+1} ^n b_m \,\Phi(\zeta_m,\theta_m) \\
&= \frac1{\bar b_n} \sum_{m=n_0(n)+1} ^n b_m \Bigl(\Phi(\bar \zeta_n,\bar \theta_n) + D\Phi(\bar \zeta_n,\bar \theta_n)\begin{pmatrix} \zeta_m -\bar \zeta_n \\ \theta_m-\bar \theta_n\end{pmatrix}+ \cO\bigl( |\zeta_m -\bar \zeta_n|^{1+\alpha_\Phi}+| \theta_m-\bar \theta_n|^{1+\alpha_\Phi}\bigr)\Bigr)\\
&=\Phi(\bar \zeta_n,\bar \theta_n)+ \cO\Bigl(\frac1{\bar b_n} \sum_{m=n_0(n)+1} ^n b_m \bigl(|\zeta_m -\bar \zeta_n|^{1+\alpha_\Phi}+| \theta_m-\bar \theta_n|^{1+\alpha_\Phi}\bigr)\Bigr),
\end{align*}
where we used convexity of $U_\Phi$ and linearity of $D\Phi(\bar \zeta_n,\bar \theta_n)$. 
	With Proposition~\ref{prop:zeta_est} we get that
	$$
		\sup_{m=n_0(n)+1,\dots,n } |\zeta_m-\zeta_{n_0(n)}| = \cO_P\bigl( \eps_n^{(\ref{prop:zeta_est})}\bigr) \text{, \ on } \Uconv_{\infty},	
	$$
	so that, on $\Uconv_\infty$,
	\begin{align*}
\frac{1}{\bar b_n} \sum_{m=n_0(n)+1}^{n} b_m | \zeta_m- \bar \zeta_n|^{1+\alpha_\Phi} \le \Bigl(2 \sup_{m=n_0(n)+1,\dots,n } |\zeta_m-\zeta_{n_0(n)}|\Bigr)^{1+\alpha_\Phi}= \cO_P\bigl(( \eps_n^{(\ref{prop:zeta_est})})^{1+\alpha_\Phi}\bigr).
	\end{align*}
	By assumption~(\ref{assu:eps2}), the previous expression is of order $o_P(\sigma_n)$.
Moreover, using that $|a-b|^{1+\alpha_\Phi}\leq (|a|+|b|)^{1+\alpha_\Phi}\leq 2^{\alpha_\Phi} (|a|^{1+\alpha_\Phi}+|b|^{1+\alpha_\phi})$ for $a,b\in \R^{d_\theta}$, $\sum^n_{m=n_0(n)+1}b_m=\bar b_n$ and Jensen's inequality we conclude that, on $\IU_\infty$,
\begin{align}\begin{split}\label{eq9456643}
\frac 1{\bar b_n} \sum_{m=n_0(n)+1}^n b_m & |\theta_m-\bar \theta_n|^{1+\alpha_\Phi}\leq  2^{\alpha_\Phi}\frac 1{\bar b_n} \sum_{m=n_0(n)+1}^n b_m  (|\theta_m|^{1+\alpha_\Phi} +  |\bar\theta_n|^{1+\alpha_\Phi})\\
&\leq  2^{1+\alpha_\Phi}\frac 1{\bar b_n} \sum_{m=n_0(n)+1}^nb_m |\theta_m|^{1+\alpha_\Phi}\leq  2^{1+\alpha_\Phi}\Bigl( \frac 1{\bar b_n} \sum_{m=n_0(n)+1}^nb_m |\theta_m|^2\Bigr)^{(1+\alpha_\Phi)/2}.
\end{split}\end{align}
Recall that, on $\IU_\infty$, $|\theta_m|=d(X_m,M)$ so that the bound of Theorem~\ref{thm:L2_bound} implies that
$$
\E\Bigl[ \1_{\IU_\infty}	 \frac 1{\bar b_n} \sum_{m=n_0(n)}^nb_m |\theta_m|^2\Bigr]   = \cO\Bigl(	 \frac 1{\bar b_n} \sum_{m=n_0(n)}^nb_m (\sigma^\mathrm{RM}_m)^2 \Bigr).
$$
so that by~(\ref{eq9456643}), on $\IU_\infty$,
$$
\frac 1{\bar b_n} \sum_{m=n_0(n)+1}^n b_m |\theta_m-\bar \theta_n|^{1+\alpha_\Phi}=\cO_P\Bigl(\Bigl(  \frac 1{\bar b_n} \sum_{m=n_0(n)+1}^nb_m (\sigma^\mathrm{RM}_m)^2 \Bigr)^{(1+\alpha_\Phi)/2}\Bigr).
$$
Hence, this term is of order $o_P(\sigma_n)$, on $\IU_\infty$, by assumption~(\ref{assu:87456}). 
Altogether, we thus get that
$$
\bar X_n=\Phi(\bar \zeta_n,\bar \theta_n)+o_P(\sigma_n)\text{, \ on $\Uconv_\infty$}.
$$

7) \emph{Synthesis.}
Note that on $\Uconv_\infty$, from a random minimal $n$ onwards all  $\bar X_n$ lie in $U_\delta^\rho$ and $\Psi$ is Lipschitz on $U_\delta^\rho$, since it has regularity $\alpha_\Psi$, so that we get with step 6 that
$$
\Psi(\bar X_n)=\begin{pmatrix} \bar \zeta_n\\ \bar \theta_n\end{pmatrix} +o_P(\sigma_n)\text{, \ on $\Uconv_\infty$.}
$$
Consequently, by step 5, and Lemma~\ref{lem:stablyOP}, one has 
$$
 \sigma_n^{-1}\Psi_\theta(\bar X_n)  \stab A \ \cN(0, \Gamma)\text{, \ on $\Uconv_\infty$.}
$$
Now
$$
 \sigma_n^{-1}( \Psi (\bar X_n)- \Psi(\bar X_n^*))= \begin{pmatrix} 0 \\  \sigma_n^{-1}\Psi_\theta(\bar X_n)\end{pmatrix}\stab \bar A \ \cN(0, \Gamma )\text{, on $\Uconv$,}
$$
with
$$
\bar A =\begin{pmatrix} 0  \\  A\end{pmatrix}= D \Psi(X_\infty) \bigl (Df(X_\infty)\big|_{N_{X_\infty} M}\bigr)^{-1} \Pi_{N_{X_\infty}M}
$$
Here we used that the image of $Df(X_\infty)\big|_{N_{X_\infty} M}$ is in $N_{X_\infty} M$ so that $$D \Psi_\zeta(X_\infty) \bigl (Df(X_\infty)\big|_{N_{X_\infty} M}\bigr)^{-1}\Pi_{N_{X_\infty}M}=0.$$
Next, note that, on $\Uconv_\infty$,
$$
\bar X_n - \bar X_n^* = \Phi(\Psi(\bar X_n))-\Phi(\Psi(\bar X_n^*))= D\Phi(\Psi(\bar X_n^*)) (\Psi(\bar X_n)-\Psi(\bar X_n^*))+ o(|\Psi(\bar X_n)-\Psi(\bar X_n^*)|)
$$
with $D\Phi(\Psi(\bar X_n^*))\to D\Phi(\Psi(X_\infty))$, almost surely, on $\Uconv_\infty$. Hence, $\sigma_n^{-1} D\Phi(\bar X_n^*) (\Psi(\bar X_n)-\Psi(\bar X_n^*))$ can be viewed as continuous function of $(D\Phi(\Psi(\bar X_n^*)), \sigma_n^{-1}(\Psi(\bar X_n)-\Psi(\bar X_n^*))$ which itself converges stably, on $\Uconv_\infty$, by Lemma~\ref{lem:stabletwovar}. Moreover, the above error term is of order $o_P(\sigma_n^{-1})$, on $\Uconv_\infty$, so that with Lemma~\ref{lem:stablyOP},
$$
\sigma_n^{-1} (\bar X_n - \bar X_n^*)  \stab Q \ \cN(0, \Gamma )\text{, \ on $\Uconv_\infty$,}
$$
with
$$
Q= D\Phi(\Psi(X_\infty))  \bar A = \bigl (Df(X_\infty)\big|_{N_{X_\infty} M}\bigr)^{-1} \Pi_{N_{X_\infty}M}=B.
$$
Thus we proved~(\ref{state2}).

Finally, on $\Uconv_\infty$, for sufficiently large $n$ Taylor together with the fact that $f(\bar X_n^*)=0$ imply that
$$
	F(\bar X_n)-F(X_\infty)=\frac 12 Df(\bar X_n^*) (\bar X_n-\bar X^*_n)^{\otimes 2} + o(|\bar X_n-\bar X^*_n|^2).
$$
Moreover, using that $Df=D^2F$ is a symmetric matrix  we conclude that
\begin{align*}
Df(\bar X_n^*) (\bar X_n-\bar X^*_n)^{\otimes 2}& = (\bar X_n-\bar X^*_n)^\dagger D^2F(\bar X_n^*)(\bar X_n-\bar X^*_n)\\
&= \bigl|(D^2F(\bar X_n^*))^{1/2} (\bar X_n-\bar X^*_n)\bigr|^2.
\end{align*}
Consequently, $\sigma_n^{-2}Df(\bar X_n^*) (\bar X_n-\bar X^*_n)^{\otimes 2}$ is a continuous function of $((D^2F(\bar X_n^*))^{1/2}, \sigma_n^{-1}(\bar X_n-\bar X^*_n))$ with the first component converging, almost surely, to $(D^2F( X_\infty))^{1/2}$, on $\Uconv_\infty$, and the second component converging stably as derived above. Hence, we get stable convergence  
$$
2\sigma_n^{-2} (F(\bar X_n)-F(X_\infty)) \stab \bigl|\bigl (Df(X_\infty)\big|_{N_{X_\infty} M}\bigr)^{-1/2} \ \Pi_{N_{X_\infty}M} \ \cN(0, \Gamma)\bigr|^2
$$
which is statement~(\ref{state3}).
\end{proof}

\subsection{Proof of Theorem~\ref{thm:MainCLT_special}}
\begin{proof} 
 First we verify that for every triple $(\alpha_f,\alpha_\Phi,\alpha_\Psi)$ as in  (A.1)  there exist $\gamma$ and $\rho$ satisfying~(\ref{assu:987}) and that for every such $\gamma$ and $\rho$ there exists $(n_0(n))_{n\in\N}$ as in~(A.3). By definition, $\alpha'>\frac 12$ so that every term on the left hand side of $\gamma$ in condition~(\ref{assu:987}) is strictly smaller than one. Hence $\gamma$ and $\rho$ can be chosen accordingly.

We prove existence of a $\N$-valued sequence $(n_0(n))$ with $0\leq n_0(n)<n$, $n_0(n)=o(n)$ and
	$$
		n_0(n)^{-1} = o \Bigl(n^{-\frac 1{2\gamma-1}\frac 1{1+\alpha_\Phi}} \wedge  n^{-\frac1\alpha \frac {1-\gamma}{2\gamma-1} } \Bigr).
	$$
	With assumption~(\ref{assu:987}) we have $\gamma > (1-\frac 12 \frac{\alpha_\Phi}{1+\alpha_\Phi})\vee (1-\frac{\alpha}{1+2\alpha})$ and elementary computations imply that
	$$
	\frac{1}{2\gamma-1}\frac{1}{1+\alpha_\Phi}< 1 \text{ \ and \ } \frac 1\alpha \frac{1-\gamma}{2\gamma-1} <1.
	$$
	Hence, the choice  $n_0(n)=\lfloor n^{\beta}/2\rfloor$ with $\lfloor\cdot\rfloor$ denoting the rounding off operation fulfills  assumption~(\ref{assu:986}) when choosing 
	$$
	\beta \in\Bigl( \frac{1}{2\gamma-1}\frac{1}{1+\alpha_\Phi} \vee \frac 1\alpha \frac{1-\gamma}{2\gamma-1},1\Bigr).
	$$
	Now suppose that $\rho-\gamma < -1$. By assumption (\ref{assu:987}), we have
	\begin{align} \label{ineq:gamma}
		\gamma > \frac{1+\frac{\alpha_\Phi}{2}}{1+\alpha_\Phi}> \frac{1}{1+\alpha_\Phi}
	\end{align}	
	so that we can additionally assume that $\beta>\bigl(\frac 1{1+\alpha_\Phi}-(1+\rho)\bigr)/(\gamma-(1+\rho))$ since the right hand side is strictly smaller than one. For this choice we then also have
that
	$$
		n_0(n)^{-1}= o\Bigl(n^{-\frac{\frac 1{1+\alpha_\Phi}-(1+\rho)}{\gamma-(1+\rho)}}\Bigr).
	$$\smallskip

Next, we  verify   the assumptions of Theorem~\ref{thm:MainCLT} with $\sigma_n^\mathrm{RM}=n^{-\gamma/2}$ and $\delta_n^\mathrm{diff}\equiv 1$. Note that $\gamma>1-\frac 12 \frac{\alpha_\Phi}{1+\alpha_\Phi}$ implies that $\gamma>\frac 34$. 

{\bf (B.1)+(B.3):} Immediate consequences of the assumptions.

{\bf (B.2):} 
	By definition of $(\gamma_n)$ one has $n\gamma_n\to\infty$ and $\gamma_n\to0$. Furthermore, it is elementary to check that
	$$
	\frac{b_{n+1}\gamma_n}{b_n\gamma_{n+1}} =1+(\rho+\gamma) n^{-1} +o(n^{-1})=1+o(\gamma_n)
	$$
	since $\gamma_n=n^{-\gamma}$ with $\gamma<1$.
	
	Moreover, note that
	$$
	\frac{\sigma_{n-1}^{\mathrm{RM}}-\sigma_{n}^{\mathrm{RM}}}{\sigma_{n}^{\mathrm{RM}}}= \frac{\gamma}{2n}+ o(n^{-1})= o(\gamma_n)
	$$
	and trivially $\sigma_{n-1}^\mathrm{RM}\approx \sigma_n^\mathrm{RM}$.	
	By assumption~(\ref{assu:987}), $2\rho>2\gamma\alpha'-2>-1$.  Hence, 
	$$
	\sum_{m=n_0(n)+1} ^n (b_m \delta_m^\mathrm{diff})^2 \sim \int_{n_0(n)}^n  s^{2\rho}\, \dd s =\Bigl[\frac 1{2\rho+1} s^{2\rho+1}\Bigr]_{n_0(n)}^n\sim \frac 1{2\rho+1} n^{2\rho+1}.
	$$
	Similarly, for $(L(n))$ as in (B.2)
	$$
	\sum_{m=L(n)+1} ^n (b_m \delta_m^\mathrm{diff})^2 \sim \int_{L(n)}^n  s^{2\rho}\, \dd s =\Bigl[\frac 1{2\rho+1} s^{2\rho+1}\Bigr]_{L(n)}^n =o\bigl(  n^{2\rho+1}\bigr),
	$$
	since $L(n)^{2\rho+1}\sim n^{2\rho+1}$. Consequently, 
	$$
			\lim_{n\to\infty}\frac{\sum_{k=L(n)+1}^n (b_k \delta^{\mathrm{diff}}_k)^2}{\sum_{k=n_0(n)+1}^n (b_k \delta^{\mathrm{diff}}_k)^2}= 0.
	$$
	
	{\bf(B.4):} 
	The almost sure convergence of $(\cov(D_m|\cF_{m-1}))_{m\in\N}$ on $\Mconv$  is true by  assumption.

	Let $x \in M$. According to (A.4) we can fix an open neighbourhood $U\subset \R^d$  of $x$  such that $(\1_U(X_{n-1})|D_n|^2)_{n\in\N}$ is uniformly integrable  and denote by $\Uconv$ the event, that $(X_n)$ converges to a point in $M \cap U$.  Let $\eps,\eps'>0$ arbitrary. To verify~(\ref{assu:D_n}) we note that
		\begin{align*}
		\P \Bigl( \Bigl\{ ( \sigma_n )^{-2} &\sum_{m=n_0(n)+1}^n \frac {b_m^2}{\bar b_n^2} \E[\1_{\{|D_m|>\eps\bar b_n  \sigma_n /b_m\}}|D_m|^2|\cF_{m-1}] > \eps'\Bigr\}\cap \Uconv \Bigr) \\
		& \le \P\bigl(\bigl\{\exists m\in\{ n_0(n)+1,\dots,  n\}: X_{m-1} \notin U\bigr\}\cap \Uconv\bigr) \\
		&+ \frac 1{\eps'} \E\Bigl[( \sigma_n )^{-2} \sum_{m=n_0(n)+1}^n \frac {b_m^2}{\bar b_n^2} \E[\1_U(X_{m-1})\1_{\{|D_m|>\eps\bar b_n  \sigma_n /b_m\}}|D_m|^2|\cF_{m-1}]\Bigr]
	\end{align*}
	and we will verify that the previous two summands converge to zero as $n\to\infty$.  
	
		The first term converges to zero, since on $\Uconv$ the process stays in $U$ from a random index onwards.
	To verify that also the second term tends to zero we observe that 
	\begin{align}\label{eq:9876}
	\bar b_n =\sum_{m=n_0(n)+1}^n b_m \sim \frac{1}{\rho+1}n^{\rho+1}
	\text{ \ so that \ }
	\sigma_n = \frac{1}{\bar b_n} \sqrt{\sum_{m=n_0(n)+1} ^n (b_m \delta_m^\mathrm{diff})^2} \sim \frac{\rho+1}{\sqrt{2\rho+1}}n^{- 1/2} \to 0
	\end{align}
	and
	$$
		 \bar b_n \sigma_n \sim \frac{1}{\sqrt{2\rho+1}}n^{\rho+\frac{1}{2}} 
	\text{ \ 	entails that \ }
		\inf\limits_{m=n_0(n)+1,\dots,  n} \bar b_n \sigma_n/b_m   \to \infty \text{ \ as $n \to \infty$,}
	$$
	since   $\rho >-\frac 12$ and $b_m=m^\rho$.
Hence, by  the uniform integrability of $(\1_U(X_{m-1}) |D_m|^2)_{m\in\N}$ we get that
	\begin{align} \label{eq:900004}
		\sup\limits_{m=n_0(n)+1,\dots, n} \E[\1_U(X_{m-1})\1_{\{|D_m|>\eps\bar b_n  \sigma_n /b_m\}}|D_m|^2] \to 0.
	\end{align}
	and with $\sigma_n \to\infty$ we arrive at 
		\begin{align*}
	\lim_{n\to\infty} ( \sigma_n )^{-2} \sum_{m=n_0(n)+1}^n \frac {b_m^2}{\bar b_n^2} \E[\1_U(X_{m-1})\1_{\{|D_m|>\eps\bar b_n  \sigma_n /b_m\}}|D_m|^2]=0, 
	\end{align*}
	so that we established  convergence to zero in probability on $\Uconv$. Similarly to \ref{le:feasible} there exists a countable family $\cU$ of open sets such that $(\1_U(X_{n-1})|D_n|^2)$ is uniformly integrable for all $U \in \cU$ and
	$$
		M \subset \bigcup\limits_{U \in \cU} U.
	$$
	By the above argument (\ref{assu:D_n}) holds on each $\Uconv$ with $U\in \cU$ and hence also on 
	$$
	\Mconv=\bigcup_{U\in \cU} \Uconv.
    $$
	%
	
	Assumption~(\ref{assu:D_n2}) is true since $\sqrt{\gamma_n}/\sigma_n^\mathrm{RM}=\sqrt{C_\gamma}$ and 
	$(\E[\1_U(X_{m-1}) |D_m|^2])_{m\in\N}$ is uniformly bounded by uniform integrability.
	
	The other assumptions of {\bf(B.4)} are immediate consequences of {\bf(A.4)} and the fact that $\delta_n^\mathrm{diff}\equiv 1$ and $\sigma_n^\mathrm{RM}=n^{-\gamma/2}$.

	{\bf(B.5):} 
	Using that $\bar b_n \sim \frac1{\rho+1} n^{\rho+1}$ we conclude that
	$$
	\sigma_n^{-1}\frac {b_{n_0(n)}}{\bar b_n \gamma_{n_0(n)}}\, \sigma_{n_0(n)}^\mathrm{RM} \sim \frac{\sqrt{2\rho+1}}{C_\gamma}n^{\frac 12 -(\rho+1)} n_0(n)^{\rho+\gamma-\frac \gamma2} =\frac{\sqrt{2\rho+1}}{C_\gamma}\frac { n_0(n)^{\rho+\frac \gamma2}} {n^{\rho+\frac 12}}
	$$
which tends to zero since, by  assumption~(\ref{assu:987}),  $\rho+\frac 12>\gamma\alpha'-\frac 12 > 0$, $\gamma<1$ and $n_0(n)\leq n$.
	
	We verify that $(\eps_n^{(\ref{prop:zeta_est})})^{1+\alpha_\Phi}=o(\sigma_n)$.
	\begin{align*}
		\eps_n^{(\ref{prop:zeta_est})}&=\sum_{m=n_0(n)+1}^n ((\sqrt{ \gamma_m}\sigma_m^\mathrm{RM})^{1+\alpha_\Psi} + \gamma_m (\sigma_{m-1}^\mathrm{RM})^{1+\alpha} )+ \sqrt{\sum_{m=n_0(n)+1}^n \gamma_m (\sigma_m^\mathrm{RM})^{2}}\\
		&\sim \sum_{m=n_0(n)+1}^n ( C_\gamma^{\frac{1+\alpha_\Psi}{2}}m^{-\gamma(1+\alpha_\Psi)} +C_\gamma m^{-\gamma (1+\frac {1+\alpha}2)}) + \sqrt{\sum_{m=n_0(n)+1}^n C_\gamma m^{-2\gamma}}\\
		&=\cO  \Big( \sum_{m=n_0(n)+1}^n  m^{-\gamma(1+\alpha')} + \sqrt{\sum_{m=n_0(n)+1}^n m^{-2\gamma}} \Big)\\
		&=\cO \Big( n_0(n)^{1-\gamma(1+\alpha')}+ n_0(n)^{-\gamma+\frac 12} \Big) ,
	\end{align*}
	where we used that  $\gamma(1+\alpha')$ and $2\gamma$ are strictly bigger than $1$ since $\gamma>\frac 34$ and $\alpha'>\frac 12$.
By assumption $\gamma > \frac{1}{2\alpha'}$ so that  $1-\gamma(1+\alpha')< -\gamma+\frac 12$ and $\eps_n^{(\ref{prop:zeta_est})} =\cO\bigl( n_0(n)^{- \gamma+\frac 12}\bigr)$. With~(\ref{assu:986}) we thus get that
	$$
	(\eps_n^{(\ref{prop:zeta_est})})^{1+\alpha_\Phi}= \cO \bigl( \bigl(n_0(n)^{-1}\bigr)^{(\gamma-\frac 12)(1+\alpha_\Phi)}\bigr) = o \bigl(n^{-\frac 12} \bigr),
	$$
	which is by~(\ref{eq:9876}) of order $o(\sigma_n)$.
	
	We verify that $\eps_n^{(\ref{prop:tec_2})}=o(\sigma_n)$. One has by definition of $\alpha'$
	\begin{align*}
		\eps_n^{(\ref{prop:tec_2})} &= \frac 1{\bar b_n} \sum_{m=n_0(n)+1}^n b_m \bigl(\gamma_m^{-\frac {1-\alpha_\Psi}2} (\sigma_m^\mathrm{RM})^{1+\alpha_\Psi}+(\sigma_{ m-1}^\mathrm{RM})^{1+\alpha}+\sigma_{ m-1}^\mathrm{RM}(\eps_n^{(\ref{prop:zeta_est})})^\alpha\bigr)\\
		&=\cO\Bigl( \frac {1}{n^{\rho+1}} \sum_{m=n_0(n)+1}^n m^\rho \bigl(\underbrace{m^{\gamma\frac {1-\alpha_\Psi}2-\frac\gamma2 (1+\alpha_\Psi) }}_{=m^{-\gamma \alpha_\Psi}} +m^{-\gamma \frac{1+\alpha}2}+ m^{-\frac\gamma2} n_0(n)^{-\alpha(\gamma-\frac 12)}\bigr)\Bigr)\\
		&=\cO\Bigl( \frac 1{n^{\rho+1}} \sum_{m=n_0(n)+1}^n m^\rho \bigl(m^{-\gamma \alpha'}+m^{-\frac\gamma2} n_0(n)^{-\alpha(\gamma-\frac 12)}\bigr)\Bigr)\\
		&=\cO \big( n^{-\gamma \alpha'} +n^{-\frac \gamma2} n_0(n)^{-\alpha(\gamma-\frac 12)}\big),
	\end{align*}
where  we used that $\rho-\gamma \alpha'>-1$ and $\rho-\frac\gamma2>-1$ as consequence of (\ref{assu:987}).
	Recall that by assumption $\gamma \alpha'>\frac 12$ and $n_0(n)^{-1}=o(n^{-\frac 1\alpha \frac {1-\gamma}{2\gamma-1}})$ so that $\eps_n^{(\ref{prop:tec_2})}=o(n^{-\frac 12}) =o(\sigma_n)$.
	%
	%

	Finally, we show that 
	$$
		\frac{1}{\bar b_n} \sum_{m=n_0(n)+1}^{n} b_m (\sigma_m^{\mathrm{RM}})^2= o(n^{-\frac{1}{1+\alpha_\Phi}}).
	$$
	We have
	$$
		\frac{1}{\bar b_n} \sum_{m=n_0(n)+1}^{n} b_m (\sigma_m^{\mathrm{RM}})^2 \sim \frac{\rho+1}{n^{\rho+1}} \sum_{m=n_0(n)+1}^{n} n^{\rho-\gamma}
	$$
	 so that in the case where $\rho - \gamma >-1$ the latter term is of order $O(n^{-\gamma})=o(n^{-\frac{1}{1+\alpha_\Phi}})$ as consequence of  (\ref{ineq:gamma}). 	
	 In the case where $\rho-\gamma=1$ we use that $\rho+1=\gamma > 1/(1+\alpha_\Phi)$  to conclude that
	 $$
	 	\frac{1}{n^{\rho+1}}\sum_{m=n_0(n)+1}^{n} m^{-1}\le \frac{1}{n^{\rho+1}} \log (n) = o(n^{-\frac{1}{1+\alpha_\Phi}}).
	 $$ 
	Finally, in the case where  $\rho-\gamma<-1$ with (\ref{assu:984})
	$$
		\frac{1}{n^{\rho+1}}\sum_{m=n_0(n)+1}^{n} m^{\rho-\gamma} = \cO\Bigl(\frac{n_0(n)^{-\gamma+\rho+1}}{n^{\rho+1}}\Bigr) =o(n^{-\frac{1}{1+\alpha_\Phi}}).
	$$
\end{proof}

\section{Appendix} \label{sec:appendix}

\subsection{Stable convergence} \label{app:stable}

In this section, we introduce the concept of stable convergence on a set. It is a slight generalisation of stable convergence introduced in~\cite{Ren63}. 

\begin{defi}Let $(Y_n)_{n\in\N}$ be a sequence of $\R^d$-valued random variables, $A\in\cF$ and $K$ a probability kernel from $(A, \cF|_A)$ to $(\R^d,\cB^d)$. We say that $(Y_n)$ \emph{converges stably on $A$ to $K$} and write
	$$
	Y_n \stackrel{\mathrm{stably}}\Longrightarrow K, \text{ \ on $A$},
	$$
	if for every $B\in\cF$ and continuous and bounded function $f:\R^d\to \R$
	\begin{align}\label{eq:def_stable}
		\lim_{n\to\infty} \E\bigl[\1_{A\cap B} f(Y_n)\bigr] = \E\Bigl[\1_{A\cap B} \int f(y) \,K(\cdot ,dy)\Bigr].
	\end{align}
	In the case where $A=\Omega$, we briefly say that $(Y_n)$ \emph{converges stably  to $K$} and write
	$$
	Y_n \stackrel{\mathrm{stably}}\Longrightarrow K.
	$$
\end{defi}

We give some central properties of stable convergence.
\begin{theorem} Let $(Y_n)$, $A$ and $K$ as in the previous definition and let $\cE$ denote a $\cap$-stable generator of $\cF$ containing $\Omega$. The following properties are equivalent.
	\begin{enumerate}
		\item[(i)] $(Y_n)$ converges stably to $K$ on $A$.
		\item[(ii)] For every  $B\in\mathcal E$ and continuous and bounded function $f:\R^d\to \R$
		$$
		\lim_{n\to\infty} \E\bigl[\1_{A\cap B} f(Y_n)\bigr] = \E\Bigl[\1_{A\cap B}\int f(y) \,K(\cdot ,dy)\Bigr].
		$$
		\item[(iii)] For every $B\in\cE$ and $\xi\in\R^d$
		$$
		\lim_{n\to\infty} \E\bigl[\1_{A\cap B} e^{i\langle \xi, (Y_n)\rangle}\bigr] = \E\Bigl[\1_{A\cap B}\int e^{i\langle \xi, y\rangle} \,K(\cdot ,dy)\Bigr].
		$$
		\item[(iv)] For every bounded random variable $\Upsilon$ and every bounded and continuous $f:\R^d\to\R$
		$$
		\lim_{n\to\infty} \E\bigl[\1_{A} \Upsilon f(Y_n)\bigr] = \E\Bigl[\1_A \Upsilon \int  f(y) \,K(\cdot ,dy)\Bigr].
		$$
	\end{enumerate}
\end{theorem}

\begin{proof}${\bf (ii)\Rightarrow(i):}$ First suppose that $f:\R^d\to \R$ is nonnegative. It is standard to verify that the set $\cF^f$ of all sets $B\in \cF$ with the property that 
	$$
	\lim_{n\to\infty} \E\bigl[\1_{A\cap B} f(Y_n)\bigr] = \E\Bigl[\1_{A\cap B}\int f(y) \,K(\cdot ,dy)\Bigr]
	$$
	is a Dynkin-system.  Since $\cF^f$ contains the generator $\cE$ we thus have $\cF^f=\cF$ and we  verified property~(\ref{eq:def_stable}) for  nonnegative  $f:\R^d \to\R$.
	For a general bounded and continuous function $f:\R^d\to \R$ we write $f=\bar f-c$ with a nonnegative function $\bar f:\R^d\to\R$ and a constant $c\geq 0$. Clearly, (\ref{eq:def_stable}) holds for $\bar f$ and the constant function $c$ and by linearity of the integral  and the limit we get that~(\ref{eq:def_stable}) also  holds for $f=\bar f-c$.\smallskip
	
	${\bf (iii)\Rightarrow (ii):}$  
	Follows from \cite[Cor 3.8]{Hau15} where we set in the notation of the corollary $\mathcal{G}= \cF|_A$ with the $\cap$-stable generator $\{A \cap B | B \in \cE\}$.

	${\bf (i)\Rightarrow (iv):}$ For  nonnegative $f$ and $\Upsilon$, the asymptotic property follows by a monotone class argument and the general case is derived by using linearity.
\end{proof}

\begin{lemma} \label{lem:stableonsubsets}
	\begin{enumerate} \item Let $A,A'\in\cF$ and suppose that $(Y_n)$ converges stably to $K$ and $K'$ on $A$ and $A'$, respectively. Then for almost all $\omega\in A\cap A'$ one has
		$$
		K(\omega,\cdot) =K'(\omega, \cdot).
		$$
		In particular, the kernel appearing as limit is unique up to almost sure equivalence.
		\item Let $(A_m)_{m\in\N}$ be a subfamily of $\cF$ and suppose that for each $m\in\N$, $(Y_n)$ converges stably to $K_m$ on $A_m$. Then 
		there exists a probability kernel $K$ from $A:=\bigcup_{m\in\N} A_m$ to $\R^d$ such that for all $m\in\N$ and almost all $\omega\in A_m$
		$$
		K(\omega,\cdot)=K_{m}(\omega,\cdot)
		$$
		and for every such kernel $K$  we have
		$$
		Y_n\stackrel{\mathrm{stably}}{\Longrightarrow} K\text{, \ on }A.
		$$
	\end{enumerate}
\end{lemma}
\begin{proof}(1): We first show uniqueness of stable limits. By basic measure theory, there exists a countable set of bounded and continuous functions $f_n:\R^d\to\R$ $(n\in\N)$ that characterize a probability distribution on $\R^d$. That means for two distributions $\mu$ and $\mu'$ on $\R^d$ one has the equivalence
	$$
	\mu=\mu'  \ \ \Longleftrightarrow \ \ \forall n\in\N: \int f_n \, d\mu= \int f_n\, d\mu'.
	$$
	Suppose now that  $(Y_n)$ converges to $K$ and $K'$ on a set $A\in\cF$. Let $n\in\N$ and 
	$$
	B^+_n=\Bigl\{\omega\in A: \int f_n(y)\, K(\omega, dy)>\int f_n(y)\, K'(\omega, dy)\Bigr\}.
	$$
	Then 
	\begin{align*} \E\Bigl[\1_{ B^+_n} \int f_n(y) \,K(\cdot ,dy)\Bigr]
		\leftarrow \E\bigl[\1_{B^+_n} f_n(Y_n)\bigr] \to  \E\Bigl[\1_{ B^+_n} \int f_n(y) \,K'(\cdot ,dy)\Bigr]
	\end{align*}
	so that
	$$
	\E\Bigl[\1_{ B_n^+} \Bigl(\int f_n(y) \,K(\cdot ,dy)- \int f_n(y) \,K'(\cdot ,dy) \Bigr)\Bigr]=0
	$$
	and $B_n^+$ is a nullset. With the same argument we obtain that the event defined as $B_n^+$ with $>$ replaced by $<$, say $B_n^-$ is a nullset. Consequently, $B=\bigcup B_n^+ \cup \bigcup B_n^-$, is a nullset and for every $\omega\in A\backslash B$ we have $K(\omega, \cdot) =K'(\omega, \cdot )$ due to the choice of $(f_n:n\in\N)$. 
	
	Now suppose that  $K$ and $K'$ are the stable limits of $(Y_n)$ on two distinct sets $A$ and $A'$, respectively. 
	As one easily verifies the restrictions of $K$ and $K'$ to $A\cap A'$ are stable limits of $(Y_n)$ on $A\cap A'$ and thus they agree by the first part up to almost sure equivalence.\smallskip
	
	(2) We first define a kernel $K$ and verify that it is the stable limit on $A$.
	Note that $A'_m:=A_m\backslash\bigcup_{k=1}^{m-1}A_k$ defines a partition $(A'_m)_{m\in\N}$ of $A$ and set for $\omega\in A$ 
	\begin{align}\label{eq83456}
		K(\omega,\cdot)=\sum_{m\in\N}\1_{A'_m} K_m(\omega,\cdot).
	\end{align}
	Fix $B\in\cF$ and a bounded and continuous function $f:\R^d\to \R$. We set $B_m= B\backslash \bigcup_{k=1}^{m-1}A_k$ and use stable convergence to $K_m$ on $A_m$ to conclude that
	\begin{align*}
		\E\bigl [\1_{A'_m\cap B} f(Y_n) \bigr]   = \E\bigl [\1_{A_m\cap B_m} f(Y_n)\bigr] & \to \E\Bigl[\1_{A_m\cap B_m} \int f(y)\, K_m(\cdot,dy)\Bigr]\\
		&\quad = \E\Bigl[\1_{A'_m\cap B} \int f(y)\, K_m(\cdot,dy)\Bigr].
	\end{align*}
	Now dominated convergence implies that
	\begin{align*}
		\E\bigl [\1_{A\cap B} f(Y_n) \bigr]  =\sum_{m\in\N} \E\bigl [\1_{A'_m\cap B} f(Y_n) \bigr]  &\to \sum_{m\in\N} \E\Bigl[\1_{A'_m\cap B} \int f(y)\, K_m(\cdot,dy)\Bigr]\\
		&\quad = \E\Bigl[\1_{A\cap B} \int f(y)\, K(\cdot,dy)\Bigr],
	\end{align*}
	where the integrable majorant is given by $(C \,\P(A_m'))_{m\in\N}$ with $C>0$ being a uniform bound for~$f$. We thus showed stable convergence on $A$ to the particular kernel $K$. Note that the previous arguments also apply for any kernel $K$ with the property that for all $m\in\N$ and almost all $\omega\in A_m$, $K(\omega,\cdot)=K_m(\omega,\cdot)$. It thus remains to show that the particular kernel possesses the latter property. However, this is an immediate consequence of part (1) since $(Y_n)$ converges stably to $K|_{A_m}$ on $A_m$ so that $K|_{A_m}$ and $K_m$ agree up to nullsets.
\end{proof}

\begin{lemma} \label{lem:stabletwovar}
	Let $d'\in\N$ and $(X_n)$ be a sequence of $\R^{d'}$-valued random variables that converges, in probability, on $A$, to a $\R^{d'}$-valued random variable $X_\infty$. If $(Y_n)$ converges stably to $K$ on $A$, then the extended sequence $(X_n,Y_n)_{n\in\N}$ converges stably, on $A$ to the kernel
	$$
	\bar K(\omega,d (x,y)) = \delta_{X_\infty(\omega)}(dx)\, K(\omega,dy).
	$$
\end{lemma}

\begin{proof}
	Choosing $\mathcal{G} = \cF|_A$, $Y= X_\infty \1_A$, $Y_n=X_n \1_A$ and $(X_n)=(Y_n)$ in Thm.~3.7 of \cite{Hau15} yields
	$$
	(\1_A X_n,Y_n) \stackrel{\mathrm{stably}}{\Longrightarrow} \delta_{\1_A X_\infty} \otimes K \text{, \ on }A,
	$$
	so that for every $B \in \cF$ and continuous and bounded function $f:\R^d\times \R^{d'}\to \R$
	$$
	\lim_{n\to\infty} \E\bigl[\1_{A\cap B} f( X_n,Y_n)\bigr] = \E\Bigl[\1_{A\cap B} \int f(x,y) \,\delta_{X_\infty}(dx)K(\cdot ,dy)\Bigr].	$$
\end{proof}

We will use a  classical central limit theorem for martingales, see~\cite{HaHe80}. 
A consequence of \cite[Corollary 3.1]{HaHe80} is the following theorem. In contrast to the original version the statement allows multidimensional processes. However, this generalisation is easily obtained by noticing that it suffices to prove the central limit theorem for linear functionals of the process.

\begin{theorem}\label{thm:martingalecentrallimit_clas}
	For every $n\in\N$ let $(Z^{(n)}_i)_{i=1,\dots,k_n}$  be a sequence of $\R^d$-valued martingale differences for a filtration $(\cF_i^{(n)})_{i=1,\dots,k_n}$ with $\cF_i^{(n)}\subset \cF_i^{(n+1)}$ for all $i=1,\dots,k_n$.  Suppose that the following holds:
	\begin{enumerate}
		\item[(i)] ${\displaystyle \forall \eps>0: \ \sum_{i=1}^{k_n} \E\bigl[\1\{|Z_i^{(n)}|>\eps\} \, |Z_i^{(n)}|^2  \big|\cF_{i-1}^{(n)}\bigr]\to 0\text{, \ in probability},}$
		\item[(ii)] ${\displaystyle  \sum_{i=1}^{k_n} \cov(Z_i^{(n)}|\cF_{i-1}^{(n)})\to \Gamma \text{, \ in probability.}}$
	\end{enumerate} 
	Then
	$$
	\sum_{i=1}^{k_n} Z_i^{(n)}\stackrel{\mathrm{stably}}\Longrightarrow \cN(0, \Gamma).
	$$
\end{theorem}

We extend the theorem  to  restricted stable convergence.

\begin{theorem}\label{thm:martingalecentrallimit}For every $n\in\N$, let $(Z^{(n)}_i)_{i=1,\dots,k_n}$  be a sequence of $\R^d$-valued martingale differences for a fixed  filtration $(\cF_i)_{i\in\N}$ and let $A\in\cF_\infty= \bigvee _{i\in\N} \cF_i$.  Suppose that  $\lim_{n\to\infty} k_n=\infty$ and the following holds:
	\begin{enumerate}
		\item[(i)] ${\displaystyle \forall \eps>0: \ \sum_{i=1}^{k_n} \E\bigl[\1\{|Z_i^{(n)}|>\eps\}\,|Z_i^{(n)}|^2  \big|\cF_{i-1}\bigr]\to 0\text{, \ in probability, on $A$},}$
		\item[(ii)] ${\displaystyle  \sum_{i=1}^{k_n} \cov(Z_i^{(n)}|\cF_{i-1})\to \Gamma \text{, \ in probability, on $A$.}}$
	\end{enumerate} 
	Then
	$$
	\sum_{i=1}^{k_n} Z_i^{(n)}\stackrel{\mathrm{stably}}\Longrightarrow \cN(0, \Gamma), \text{ \ on $A$}.	
	$$
\end{theorem}

\begin{rem}
	In the theorem one can replace assumption (i) by the stronger assumption that there exists $q>2$ with
	$$
	\sum_{i=1}^{k_n} \E\bigl[|Z_i^{(n)}|^q \big|\cF_{i-1}^{(n)}\bigr]\to 0\text{, \ in probability, on $A$}.
	$$
	Indeed, this follows since $\1\{|Z_i^{(n)}|>\eps\}\,|Z_i^{(n)}|^2  \leq \eps^{-(q-2)} |Z_i^{(n)}|^q$.
\end{rem}

\begin{proof} Applying a diagonalisation argument on property (i) we deduce existence of two zero sequences $(\delta_n)_{n\in\N}$  and $(\eps_n)_{n\in\N}$ of positive reals with
	$$
	\lim_{n\to\infty} \P\Bigl(\Bigl\{ \sum_{i=1}^{k_n} \E[ \1_{\{|Z_i^{(n)}|>\eps_n\}}|Z_i^{(n)}|^2|\cF_{i-1}] >\delta_n \Bigr\}\cap A\Bigr)=0.
	$$
	We fix $\delta\in(0,1)$ and set $I_n = \E[\1_A | \cF_n]$ for all $n \in \N$  and consider the stopping times
	$$
	T^{(n)}= \inf \Bigl\{m =0,\dots,k_n -1: I_m \le \delta \text{ \  or \ } \sum_{i=1}^{m+1} \E[ \1_{\{|Z_i^{(n)}|>\eps_n\}}|Z_i^{(n)}|^2|\cF_{i-1}] >\delta_n\Bigr \}
	$$
	with the infimum of the empty set being $\infty$.
	We will apply Theorem~\ref{thm:martingalecentrallimit_clas} onto $(\bar Z_i^{(n)})_{i=1,\dots,k_n}$ given by
	$$
	\bar Z_i^{(n)}= \1_{\{T^{(n)}\geq i\}}\,Z_i^{(n)}.
	$$
	We verify assumptions (i) and (ii). First note that for every $\eps>0$ there exists $n_0\in\N$ such that for all $n\geq n_0$, $\eps_n\leq \eps$ and for those $n$ we get that
	\begin{align*}
		\sum_{i=1}^{k_n} \E\bigl[\1\{|\bar Z_i^{(n)}|>\eps\}\,|\bar Z_i^{(n)}|^2  \big|\cF_{i-1}\bigr]&\leq \sum_{i=1}^{k_n} \E\bigl[\1\{|\bar Z_i^{(n)}|>\eps_n\}\,|\bar Z_i^{(n)}|^2  \big|\cF_{i-1}\bigr]\\
		& =\sum_{i=1}^{k_n} \1_{\{T^{(n)}\geq i\}} \E\bigl[\1\{|Z_i^{(n)}|>\eps_n\}\,| Z_i^{(n)}|^2  \big|\cF_{i-1}\bigr]\leq \delta_n \to 0.
	\end{align*}
	Second,  $(I_n)_{n \in \N}$ is a  martingale that converges to $\E[\1_A|\cF_\infty]=\1_A$, a.s., so that up to nullsets  $A^{(\delta)}:=\{\min_{n\in\N} I_n>\delta\}\subset A$. Furthermore, $\P(A^{(\delta)}\Delta \{T^{(n)}=\infty\})\to 0$
	as $n\to \infty$.   Thus we have, with high probability, on $A^{(\delta)}$,
	\begin{align*}
		\sum_{i=1}^{k_n} \cov(\bar Z_i^{(n)}|\cF_{i-1}^{(n)}) = \sum_{i=1}^{k_n} \1_{\{T^{(n)}\geq i\}} \cov(Z_i^{(n)}|\cF_{i-1})	=\sum_{i=1}^{k_n} \cov(Z_i^{(n)}|\cF_{i-1})\to \Gamma.
	\end{align*}
	Conversely, on $(A^{(\delta)})^c$ the stopping time $T=\inf\{m\in\N: I_m \le \delta\}$ is finite and we get on $(A^{(\delta)})^c$
	\begin{align*}
		\sum_{i=1}^{k_n} \|\cov(\bar Z_i^{(n)}|\cF_{i-1})\| &\leq  \sum_{i=1}^{k_n} \1_{\{T^{(n)}\geq i\}} \E[|Z_i^{(n)}|^2|\cF_{i-1}] \leq  \sum_{i=1}^{k_n} \1_{\{T^{(n)}\geq i\}} \E\bigl[|Z_i^{(n)}|^2\big|\cF_{i-1}\bigr]\\
		&\leq    \sum_{i=1}^{k_n} \1_{\{T^{(n)}\geq i\}} \bigl( \E\bigl[ \1\{|Z_i^{(n)}|>\eps_n\} |Z_i^{(n)}|^2\big|\cF_{i-1}\bigr]+\eps_n\bigr)\\
		&\leq  \bigl(\delta_n+T\eps_n)\to0.
	\end{align*}
	Thus we showed that
	$$
	\sum_{i=1}^{k_n} \bar Z_i^{(n)}\stackrel{\mathrm{stably}}\Longrightarrow \cN(0,\1_{A^{(\delta)}} \Gamma).
	$$
	Recalling that on $A^{(\delta)}$, with high probability, $\sum_{i=1}^{k_n} Z_i^{(n)}= \sum_{i=1}^{k_n} \bar Z_i^{(n)}$ we conclude that
	$$
	\sum_{i=1}^{k_n} Z_i^{(n)}\stackrel{\mathrm{stably}}\Longrightarrow \cN(0, \Gamma), \text{ \ on $A^{(\delta)}$}.	
	$$
	Finally, we note that $(I_n)$ takes values in $[0,1]$ and once the process hits zero it stays there, almost surely. Hence one  has $A=\{\min_{n\in\N} I_n>0\}$ up to nullsets. This implies that up to nullsets
	$$
	A=\bigcup_{\delta>0} A^{(\delta)}
	$$
	Thus an application of Lemma~\ref{lem:stableonsubsets} finishes the proof.
\end{proof}

%


\subsection{$\cO_P$ and $o_P$}\label{sec:8_2}

We will use the $\cO$- and $o$-notation in a probabilistic sense.

\begin{defi}
	Let $A \in \cF$, $(X_n)$ be a sequence of $\R^d$-valued random variables and $(a_n)$ be a sequence of strictly positive reals. 	
	\begin{enumerate} \item If 
		$$ \lim_{C\to\infty} \limsup_{n\to\infty}  \P(\{ |X_n|> C a_n\} \cap A)=0,
		$$
		we say that $(X_n)$ is of \emph{order $\cO (a_n)$, in probability, on $A$,} and write
		$$
		X_n = \cO_P(a_n) \text{, \ on }A.
		$$
		\item If for every $C>0$
		$$ \limsup_{n\to\infty}  \P(\{ |X_n|> C a_n\} \cap A)=0,
		$$
		we say that $(X_n)$ is of \emph{order $o(a_n)$, in probability, on $A$,} and write
		$$
		X_n = o_P(a_n) \text{, \ on }A.
		$$
	\end{enumerate}
\end{defi}

\begin{rem} \label{rem:OPMarkov}
	Expectations together with Markov's inequality are an efficient tool for verifying that a sequence  $(X_n)$ of random variables is of order $\cO(a_n)$. Indeed,  
	$$
	\limsup_{n\to\infty}  \P(\{ |X_n|> C a_n\} \cap A)\leq \frac 1C		\limsup_{n\to\infty} \frac {	\E[\1_A |X_n|]}{a_n}
	$$
	so that finiteness of the $\limsup$ on the right implies that $X_n=\cO_P(a_n)$, on $A$.
\end{rem}

	\begin{lemma} \label{lem:OPonUnion2}
		Let  $(a_n)$ be a sequence of strictly positive reals, $(X_n)$ be a sequence of $\R^d$-valued random variables and $A,A_1, A_2, \ldots \in \cF$  with $\P(A\backslash \bigcup_{m\in\N} A_m) =0$. If for every $m\in\N$
		$$
		X_n=\cO_P(a_n),\text{ \ on }A_m,
		$$
		then 
		$$
		X_n =\cO_P(a_n)\text{, \ on } A.
		$$
	\end{lemma}
	
	\begin{proof}
		Let $\eps >0$ and choose  $M \in \N$ such that $\P(A\setminus \bigcup_{m=1}^M A_{m})\leq \eps$.
		Now 
		$$ \P(\{X_n\geq C a_n\}  \cap A) \leq  \sum_{m=1}^M \P(\{X_n\geq C a_n\}  \cap A_m) + \P\Bigl(A\setminus \bigcup_{m=1}^M A_{m}\Bigr)
		$$
		so that
		$$
		\limsup _{n\to\infty}  \P(\{X_n\geq C a_n\}  \cap A) \leq  \sum_{m=1}^M  \limsup _{n\to\infty} \P(\{X_n\geq C a_n\}  \cap A_m) + \eps.
		$$
		Consequently, 
		$$
		\lim_{C\to\infty} \limsup _{n\to\infty}  \P(\{X_n\geq C a_n\}  \cap A) \leq   \eps
		$$
		and the statement follows since $\eps>0$ was arbitrary.
\end{proof}

\begin{lemma} \label{lem:stablyOP}
	Let $A \in \mathcal F$ and $(X_n),(Y_n)$ be $\R^d$-valued sequences of random variables. Suppose that $(Y_n)$ converges stably to $K$ on $A$ and $X_n=o_P(1)$, on $A$. Then
	$$
	X_n+Y_n \stackrel{\mathrm{stably}}\Longrightarrow K \text{, \ on }A.
	$$	
\end{lemma}

\begin{proof}
	Let $\eps>0$. By the assumptions on $(X_n)$ we have
	$$
	\limsup_{n \to \infty} \P(\{|X_n|>\eps\}\cap A)=0,
	$$
	so that 
	$$
	X_n \to 0 \text{, \ in probability, on $A$.}
	$$
	Thus, with Lemma~\ref{lem:stabletwovar},
	$$
	(X_n,Y_n) \stackrel{\mathrm{stably}}\Longrightarrow \delta_0 \otimes K \text{, \ on }A.
	$$
	Define $$g: \R^{d} \times \R^d \to \R^d; \quad (x,y)\mapsto x+y.$$
	Let $B \in \cF$ and $f:\R^d \to \R$ continuous and bounded. Then,
	\begin{align*}
		\E[\1_{A \cap B}f(X_n + Y_n)] = \E[\1_{A \cap B}(f\circ g)(X_n, Y_n)] &\to \E\left[\1_{A \cap B}\int \int f(x+y)\ \delta_0(dx) K(\cdot, dy)\right] \\
		&= \E\left[\1_{A \cap B}\int f(y)\  K(\cdot, dy)\right].
	\end{align*}
\end{proof}

\subsection{Nice representations in the sense of Def.~\ref{def:Phi}, Fermi coordinates} \label{sec:Paralleltransport}
In this section we discuss the existence of nice representations. 

\begin{lemma}
Let $d_\zeta\in\{1,\dots,d-1\}$ and $M\subset \R^d$ be a $d_\zeta$-dimensional $C^3$-submanifold. Then every $x\in M$ 
admits a nice representation $\Phi:U_\Phi\to U$ for a neighbourhood $U$ of $x$ that is $C^2$. 
\end{lemma}

\begin{proof}We use Fermi coordinates. Let  $U$ be an open neighbourhood of $x$ and $\Gamma:U_\Gamma \to U$ a $C^3$-diffeomorphism with
$$
\Gamma(M_\Gamma\times \{0\}^{d_\theta})= U \cap M, \ \text{ where } M_\Gamma:=\{\zeta\in\R^{d_\zeta}: (\zeta,0)\in U_\Gamma \}
$$
and $d_\theta = d-d_\zeta$. We define a mapping
$$
\tilde \Phi: M_\Gamma\times \R^{d_\theta}\to \R^d
$$
as follows.
For every $\zeta\in M_\Gamma$ we apply the Gram-Schmidt orthonormalisation procedure to the column vectors of the invertible matrix $D\Gamma (\zeta,0)$ that is the vectors $D\Gamma (\zeta,0)e_1,\dots,D\Gamma (\zeta,0)e_d$ with $e_1,\dots,e_d$ denote the standard basis of $\R^d$. That means we iteratively set for $k=1,\dots,d$
%
$$
	\bar e_k(\zeta) = \frac{D\Gamma (\zeta,0)e_k-\sum_{i=1}^{k-1} \langle \bar e_i(\zeta),D\Gamma (\zeta,0)e_k\rangle \,  \bar e_i(\zeta) }{|D\Gamma (\zeta,0)e_k-\sum_{i=1}^{k-1} \langle \bar e_i(\zeta),D\Gamma (\zeta,0)e_k\rangle \,  \bar e_i(\zeta) |}.
$$
By induction over $k$ it easily follows that the mapping $\zeta\mapsto \bar e_k(\zeta)$ is $C^2$ and we set
$$
\tilde \Phi: M_\Gamma\times \R^{d_\theta}\to \R^d, \ (\zeta,\theta) \mapsto \Gamma(\zeta,0) + \sum_{i=1}^{d_\theta} \theta_i \bar e_{d_\zeta+i}(\zeta).
$$
Note that $\tilde \Phi$ is $C^2$ and $\bar e_{d_\zeta+1}(\zeta),\dots,\bar e_{d}(\zeta)$ span the normal space $N_{\Gamma(\zeta,0)} M$. We differentiate $\tilde \Phi$ in $(\zeta,0)$ with $\zeta \in M_\Gamma$. One has  for every $k=1,\dots ,d_\zeta$ and $\ell=1,\dots,d_\theta$,
$$
\frac{\partial}{\partial  \zeta_k}\tilde \Phi (\zeta,0) =\frac{\partial}{\partial  \zeta_k}\Gamma (\zeta,0) \text{ \ and \ } \frac{\partial}{\partial  \theta_\ell}\tilde \Phi (\zeta,0) = \bar e_{d_\zeta+\ell}(\zeta).
$$
By construction the first $d_\zeta$ columns of $D\tilde \Phi(\zeta,0)$ are linearly independent and span the same linear space as the vectors $\bar e_{1}(\zeta),\dots, \bar e_{d_\zeta}(\zeta)$ so that all columns of $D\tilde \Phi(\zeta,0)$ are linearly independent and $D\tilde \Phi(\zeta,0)$ is an invertible matrix.  We set $(\zeta_0,0)=\Gamma^{-1} (x)$ and note that
 the mapping~$\tilde \Phi$ restricted to an appropriate ball $B_{r_0}(\zeta_0,0)\subset M_\Gamma\times \R^{d_\theta}$ 
 is a $C^2$-diffeomorphism onto its image.

Possibly, $\bigl(\tilde \Phi|_{B_{r_0}(\zeta_0,0)}\bigr)^{-1}(M)$ is not a subset of $\R^{d_\zeta} \times \{0\}^{d_\theta}$.
Since the manifold $M$ has no boundary we can choose $r_1\in(0,r_0)$ such that $K:=\tilde \Phi \bigl(\overline {B_{r_0}(\zeta_0,0)}\bigr)\cap M$ is compact. Hence, there exists $r_2\in(0,r_1)$ such that for all $x\in K$ and $y\in N_x M$ with $|y|\leq r_2$, $x$ is the unique closest element to $x+y$ in $M$ \cite[Theorem 3.2]{dudek1994nonlinear}. In particular, $x+y\not \in M$ if $y\not=0$. Consequently, for $(\zeta,\theta)\in B_{r_2}(\zeta_0,0)$ with $\theta\not=0$ we have
$$
\tilde \Phi(\zeta,\theta)\not \in M
$$
so that $\bigl(\tilde \Phi|_{B_{r_2}(\zeta_0,0)}\bigr)^{-1}(M)\subset \R^{d_\zeta} \times \{0\}^{d_\theta}$. Altogether, we thus proved that the restriction of $\tilde \Phi|_{B_{r_2}(\zeta_0,0)}$ is a nice representation for $M$ on $\tilde \Phi(B_{r_2}(\zeta_0,0))\ni x$.
\end{proof}
For a general introduction into Fermi coordinates of Riemannian submanifolds we refer the reader to chapter 2 of \cite{gray2012tubes}.

\bibliographystyle{alpha}
\bibliography{stoch_approx_ML}

\end{document}